\documentclass[12pt,a4paper]{article}%

\usepackage[utf8]{inputenc}
\usepackage{amsmath}
\usepackage{amsfonts}
\usepackage{amssymb}
\usepackage{amsthm}
\usepackage{bm}
\usepackage{graphicx}
\usepackage{titlesec}
\usepackage{tikz}
\usepackage[ruled, vlined]{algorithm2e}
\usepackage{subcaption}
\usepackage[export]{adjustbox}
\usepackage{placeins}
\usepackage[english]{babel}

\usepackage{booktabs}
\usepackage{multirow}

\newtheorem{theorem}{Theorem}
\newtheorem{lemma}[theorem]{Lemma}

\newtheorem{remark}[theorem]{Remark}

\newtheorem{example}{Example}

\newcommand{\RR}{\mathbb{R}}

\newcommand{\Q}{\mathcal{Q}}
\newcommand{\tQ}{\mathcal{\tilde Q}}
\newcommand{\R}{\mathcal{R}}
\newcommand{\tR}{\mathcal{\tilde R}}
\renewcommand{\P}{\mathcal{P}}
\renewcommand{\O}{\mathcal{O}}
\newcommand{\F}{\mathcal{F}}

\renewcommand{\vec}{\bf}

\renewcommand{\div}{\operatorname{div}}

\newcommand{\trih}{\mathcal{T}_h}
\newcommand{\triH}{\mathcal{T}_H}
\newcommand{\IH}{\mathcal{I}_H}
\newcommand{\Cint}{{\mathcal{I}}}

\newcommand{\average}[1]{\ensuremath{\{\hspace{-0.75ex}\{#1\}\hspace{-0.75ex}\}} }
\newcommand{\jump}[1]{\ensuremath{[\![#1]\!]} }

\newcommand{\Vf}{V^{\rm f}}

\newcommand{\tVmsk}{{\tilde V}^{\rm ms}_{k}}
\newcommand{\Vmsk}{V_k^{\rm ms}}
\newcommand{\Vms}{V^{\rm ms}}

\newcommand{\TOL}{{\rm TOL}}

\newcommand{\vf}{v^{\rm f}}
\newcommand{\uf}{u^{\rm f}}

\newcommand{\ums}{u^{\rm ms}}

\newcommand{\umsk}{{u}_k^{\rm ms}}
\newcommand{\uk}{u_k}
\newcommand{\tumsk}{{\tilde u}^{\rm ms}_k}
\newcommand{\tuk}{{\tilde u}_k}
\newcommand{\humsk}{{\hat u}^{{\rm ms}}_k}
\newcommand{\huk}{{\hat u}_k}

\newcommand{\wmsk}{{w}_k^{\rm ms}}
\newcommand{\twmsk}{{\tilde w}^{\rm ms}_k}

\title{{\bf Numerical homogenization of elliptic PDEs with similar coefficients}}
\author{Fredrik Hellman%
  \thanks{Department of Information Technology, Uppsala University, Box 337, SE-751 05 Uppsala, Sweden. Supported by Centre for Interdisciplinary Mathematics (CIM), Uppsala University.}
  \and
  Axel M\r{a}lqvist%
  \thanks{Department of Mathematical Sciences,
    Chalmers University of Technology and
    University of Gothenburg
    SE-412 96 G\"oteborg, Sweden. Supported by the Swedish Research Council.}
}

\date{}

\begin{document}

\maketitle

\begin{abstract}
  We consider a sequence of elliptic partial differential equations
  (PDEs) with different but similar rapidly varying coefficients. Such
  sequences appear, for example, in splitting schemes for
  time-dependent problems (with one coefficient per time step) and in
  sample based stochastic integration of outputs from an elliptic PDE
  (with one coefficient per sample member). We propose a
  parallelizable algorithm based on Petrov--Galerkin localized
  orthogonal decomposition (PG-LOD) that adaptively (using computable
  and theoretically derived error indicators) recomputes the local
  corrector problems only where it improves accuracy. The method is
  illustrated in detail by an example of a time-dependent two-pase
  Darcy flow problem in three dimensions.
\end{abstract}

\section{Introduction}
We consider a sequence of elliptic partial differential equations
(PDEs) with different, but in some sense similar, rapidly varying
coefficients. In some applications, the difference between consecutive
coefficients in the sequence is localized, for example for certain
Darcy flow applications and in the simulation of random defects in
composite materials. This paper studies an opportunity to exploit that
the differences are localized to save computational work in the
context of the localized orthogonal decomposition method (LOD,
\cite{MaPe14}).


The accuracy of Galerkin projection onto standard finite element
spaces generally suffers from variations in the coefficient that are
not resolved by the finite element mesh. The work \cite{BaOs83}
studies an elliptic equation in 1D with a rapidly varying coefficient
and notes that coefficient variations within the element lead to
inaccurate solutions for the standard finite element method. Replacing
the coefficient with its elementwise harmonic average leads to an
accurate method. This result, however, does not easily generalize to
higher dimensions. For periodic and semi-periodic coefficients varying
on an asymptotically fine scale, a homogenized coefficient can be
computed and used for coarse scale computations also in higher
dimensions \cite{BeLiPa78}. The early multiscale method \cite{HoWu97}
is based on homogenization theory and works under assumptions on scale
separation and periodicity. Many recent contributions within the field
of numerical homogenization can be used without assumptions on
periodicity and in higher dimensions, see e.g.\ \cite{Be16,
  HuFeGoMaQu98, KoPeYs16, MaPe14, OwZhBe14}. In this work, we consider
the LOD technique \cite{MaPe14} in the Petrov--Galerkin formulation
(PG-LOD) studied in detail in \cite{ElGiHe14}.

The fundamental idea of the LOD method is that a low-dimensional
function space (multiscale space) with good approximation properties
is constructed by computing localized fine-scale correctors to the
basis functions of a standard low-dimensional coarse finite element
space based on a coarse mesh. Each localized corrector problem is
posed only within a patch of a certain radius around its coarse basis
function and thus depends only on the diffusion coefficient in that
patch. The PG-LOD method has several good properties from a
computational perspective. The main advantage is that the PG-LOD
corrector problems can be computed completely in parallel with the
only communication being a final reduction to form a low-dimensional
global stiffness matrix. Further, the fine-scale coefficient only
needs to be accessible and stored in memory for one localized
corrector problem at a time. Additionally, the method is robust in the
sense that both the localized corrector problems and the global
low-dimensional problems are typically small enough to be solved with
a direct solver.

Once computed, the correctors can be reused for problems with the same
or similar diffusion coefficient.  We study the case when the
diffusion coefficient varies in a sequence of problems. In such
situations, there is an opportunity to reuse previously computed
localized correctors if the coefficients do not vary too much between
consecutive problems. Since the computational cost is proportional to
the number of localized corrector problems that have to be recomputed,
it is most advantageous if the perturbations of the coefficient are
localized. Two practical examples are two-phase flow where the
coefficient depends on the saturation of the two fluids, or when the
coefficient is a deviation from a base coefficient as in the case with
defects in composite materials.

In this work we derive computable error indicators for the error
introduced by refraining from recomputing a corrector after a
perturbation in the coefficient. The method we propose computes all
localized correctors and global stiffness matrix contributions for the
first coefficient in the sequence of elliptic PDEs. For the subsequent
coefficients, we use the error indicators to adaptively recompute only
the correctors that need to be recomputed in order to get a
sufficiently accurate solution.
The coefficients that have not
been recomputed we call lagging coefficients. The method is completely
parallelizable over the elements of the coarse mesh. A particularly
interesting setting is when only quantities on the coarse mesh are
required from the solution, for example upscaled Darcy fluxes in a
Darcy flow problem, or the coarse interpolation of the full
solution. Any computed fine scale quantities can then be forgotten
between the iterations in the sequence and the memory requirement
becomes very low.

The paper is divided into five sections: Problem formulation in
Section~\ref{sec:problem}, method description in
Section~\ref{sec:method}, error analysis in Section~\ref{sec:error},
implementation in Section~\ref{sec:implementation}, and numerical
experiments in Section~\ref{sec:numerical}. Both the method
description and the error analysis are divided into four steps, with
increasing level of approximation in each step: (i) reformulation by
variational multiscale method (VMS), (ii) localization by LOD, (iii)
approximation of localized correctors by lagging coefficient, and (iv)
approximation of global stiffness matrix contribution by lagging
coefficient.  The main results are the method
\eqref{eq:lod_falsespacecoef} in Section~\ref{sec:method4}, the error
bound in Theorem~\ref{thm:main} and Algorithm~\ref{alg:full}.

\section{Problem formulation}
\label{sec:problem}
Let $\Omega$ be a polygonal domain in $\RR^d$ (with $d=1, 2$ or $3$)
with the boundary partitioned into disjoint subsets $\Gamma_D$ (for
Dirichlet boundary conditions) and $\Gamma_N$ (for Neumann boundary
conditions). Suppose we have a sequence of elliptic equations: for
$n = 1, 2, \ldots$, solve for $\bar u^n$, such that
\begin{equation}
  \label{eq:mainproblem}
  \begin{aligned}
    -\div A^n \nabla \bar u^n & = f \qquad \text{in } \Omega, \\
    \bar u^n & = g \qquad \text{on } \Gamma_D, \\
    {\bf n} \cdot A^n \nabla \bar u^n & = 0 \qquad \text{on } \Gamma_N, \\
  \end{aligned}
\end{equation}
where $f \in L^2(\Omega)$, $g \in H^{1/2}(\Gamma_D)$, ${\bf n}$ is the
outward normal of the boundary, and $A^n \in L^{\infty}(\Omega)$ is a
coefficient varying significantly over small distances. To keep the
presentation short, we limit ourselves to the case where $f$ and $g$
are independent of $n$, however, the analysis in this paper can be
generalized to $n$-dependent $f$ and $g$. We will refer to the
sequence index or rank as \emph{time step} throughout the paper,
although it does not need to correspond to a step from a
time-disceratization. For instance, in
Section~\ref{sec:implementation} we briefly discuss an application for
simulation of weakly random defects in composite materials, where the
sequence index corresponds to a Monte Carlo sample member index.

In the remainder of this section and
Sections~\ref{sec:method}--\ref{sec:error}, we consider a fixed 
step $n$ and drop this index for all quantities. We call the
coefficient $A = A^n$ at the current time step the \emph{true
  coefficient}. Ideally, only the true coefficient $A$ would be used
in the solution at the current time step. However, in order to lower
the computational cost, computations from previous time steps will be
reused. This means coefficients from previous time steps
(\emph{lagging coefficients}) will enter the analysis through the
definition of the localized correctors and in the assembly of the
global stiffness matrix. These lagging coefficients will be denoted by
$\tilde A$. We also want to emphasize that the error indicators
derived here are applicable also to situations where the coefficient
deviates from a base coefficient, for example within the application
of simulations of weakly random defects in composite materials.

We will work with a weak formulation of the above problem. Let
$V = \{ v \in H^1(\Omega) \,:\, v|_{\Gamma_D} = 0\}$. In case
$\Gamma_D$ is empty, we instead consider only solutions and test
functions in the quotient space $V = H^1(\Omega) / \RR$.
Let $(\cdot, \cdot)$ denote the $L^2$-scalar product over $\Omega$,
and $(u, v)_\omega = \int_\omega uv$.  Further, we define
$\| v \|^2_{L^2(\omega)} = \int_\omega v^2$,
$\| v \|_{L^2} = \| v \|_{L^2(\Omega)}$, and the
bilinear form $a(u, v) = (A \nabla u, \nabla v)$.
We let $\bar u = u + g$, where $u \in V$ and $g \in H^1(\Omega)$ is an
extension of the boundary condition $g$ to the full domain and seek to
find $u \in V$, such that for all $v \in V$,
\begin{equation}
\label{eq:continuous}
\begin{aligned}
  (A \nabla u, \nabla v) = (f, v) - (A\nabla g, \nabla v).
\end{aligned}
\end{equation}
Assuming there exist constants $0 < \alpha$ and $\beta < \infty$, so
that $\alpha \le A \le \beta$ a.e., $(A \nabla \cdot, \nabla \cdot)$
is bounded and coercive on $V$ and existence of a unique solution is
guaranteed by the Lax--Milgram theorem. We further define the energy
norm $|\cdot|_A = (A \nabla \cdot, \nabla \cdot)^{1/2}$ on $V$, and
the semi-norm
$|\cdot|_{A,\omega} = (A \nabla \cdot, \nabla \cdot)_\omega^{1/2}$.

\section{Method description}
\label{sec:method}
In this section, we describe the proposed numerical method in a series
of steps, each of which introduces another level of approximation for
the problem \eqref{eq:continuous} above.

\subsection{Variational multiscale method}
The first step is to reformulate the problem using the variational
multiscale method \cite{HuFeGoMaQu98, HuSa07}. This formulation forms
the basis for the LOD approximation and makes it possible to reduce
the dimensionality of the problem once the corrector problems have
been solved.

Let $\triH$ be a regular and quasi-uniform family of conforming
subdivisions of $\Omega$ into elements of maximum diameter $H$, and
$V_H \subset V$ be a family of conforming first order finite element
spaces on this mesh, e.g.\ $\P_1$ or $\Q_1$ depending on the shape of the
elements. The choice of a linear projective quasi-interpolation
operator $\IH : V \to V_H$ defines the fine space as its kernel
$\Vf = \ker \IH = \{ v \in V \,:\, \IH v = 0 \}$. We assume there
exists a constant $C$ independent of $H$ so that for all
$v \in V$ and $T \in \triH$, it holds
\begin{equation}
  \label{eq:interpolation}
  H^{-1}\| v - \IH v\|_{L^2(T)} + \| \nabla (v - \IH v) \|_{L^2(T)} \le C_{\Cint} \| \nabla v \|_{L^2(U(T))}.
\end{equation}
Here $U(T)$ is the union of all neighboring elements to $T$, i.e.
\begin{equation*}
  U(T) = \bigcup \{T' \in \triH \,:\, \overline T \cap \overline {T'} \ne \emptyset \}.
\end{equation*}
Since we assume $\IH$ is projective (this is not strictly necessary,
see e.g.\ \cite{ElGiHe14}), we have the decomposition
$V = V_H \oplus \Vf$ and can decompose the solution $u = u_H + \uf$
and test function $v = v_H + \vf$ and test \eqref{eq:continuous} with
the two spaces separately:
\begin{subequations}
\begin{align}
  (A \nabla (u_H + \uf) , \nabla v_H) &= (f, v_H) - (A \nabla g, \nabla v_H), \label{eq:vms_coarse} \\
  (A \nabla \uf, \nabla \vf) &= (f, \vf) - (A \nabla g, \nabla \vf) - (A \nabla u_H, \nabla \vf). \label{eq:vms_fine}
\end{align}
\end{subequations}
We note that $\uf$ is linear in $f$ and $u_H$, and we define the
linear correction operators $\Q : H^1(\Omega) \to \Vf $ and $\R : L^2(\Omega) \to \Vf$, so
that $\uf = -\Q u_H + \R f - \Q g$, i.e., find $\Q v \in \Vf$ and
$\R f \in \Vf$, such that for all $\vf \in \Vf$,
\begin{equation}
\begin{aligned}
  \label{eq:vms_corrections}
  (A \nabla \Q v , \nabla \vf) &= (A \nabla v, \nabla \vf), \\
  (A \nabla \R f , \nabla \vf) &= (f, \vf). \\
\end{aligned}
\end{equation}
These equations have unique solutions, since
$(A\nabla \cdot, \nabla \cdot)$ is still bounded and coercive on a
subspace $\Vf \subset V$.

We introduce a new space, the multiscale space,
$\Vms = V_H - \Q V_H = \{v_H - \Q v_H\,:\,v_H \in V_H\}$, and note
that we have the orthogonality relation $\Vms \perp_a \Vf$.  The
solutions $\Q u_H$, $\R f$, and $\Q g$ can be plugged into
\eqref{eq:vms_coarse} and we get the following low-dimensional
Petrov--Galerkin problem, find $\ums \in \Vms$, such that for all
$v_H \in V_H$,
\begin{equation}
  \begin{aligned}
    \label{eq:vms_low}
    (A \nabla \ums, \nabla v_H) = (f, v_H) - (A \nabla g , \nabla v_H) - (A \nabla \R f , \nabla v_H) + (A \nabla \Q g , \nabla v_H).
  \end{aligned}
\end{equation}
The full solution is then $u = \ums + \R f - \Q g$.
\begin{remark}[Right hand side correction $\R f$]
  \label{rem:rhsc}
  It is possible to obtain an approximate solution even if neglecting
  the right hand side correction term, i.e.\ letting $\R = 0$
  above. See for example \cite{HeMa14, MaPe14}.
\end{remark}


\subsection{Localized orthogonal decomposition}
The second step is to localize the corrector computations by means of
localized orthogonal decomposition (LOD). The basic idea is to solve
the corrector problems \eqref{eq:vms_corrections} only on localized
patches instead of on the full domain to reduce the computational
cost.

For the localization, we define element patches for $T \in \triH$,
$U_k(T) \subset \Omega$, where $0 \le k \in \mathbb{N}$. With trivial
case $U_0(T) = T$, $U_k(T)$ (a $k$-layer element patch around
$T$) is defined by the recursive relation
\begin{equation*}
  U_{k+1}(T) = \bigcup \{T' \in \triH : \overline{U_{k}(T)} \cap \overline {T'} \ne \emptyset \}.
\end{equation*}
See Figure~\ref{fig:patch} for an illustration of element patches.
\begin{figure}[]
  \centering
  \begin{subfigure}{.45\textwidth}
    \centering
    \includegraphics[width=4cm]{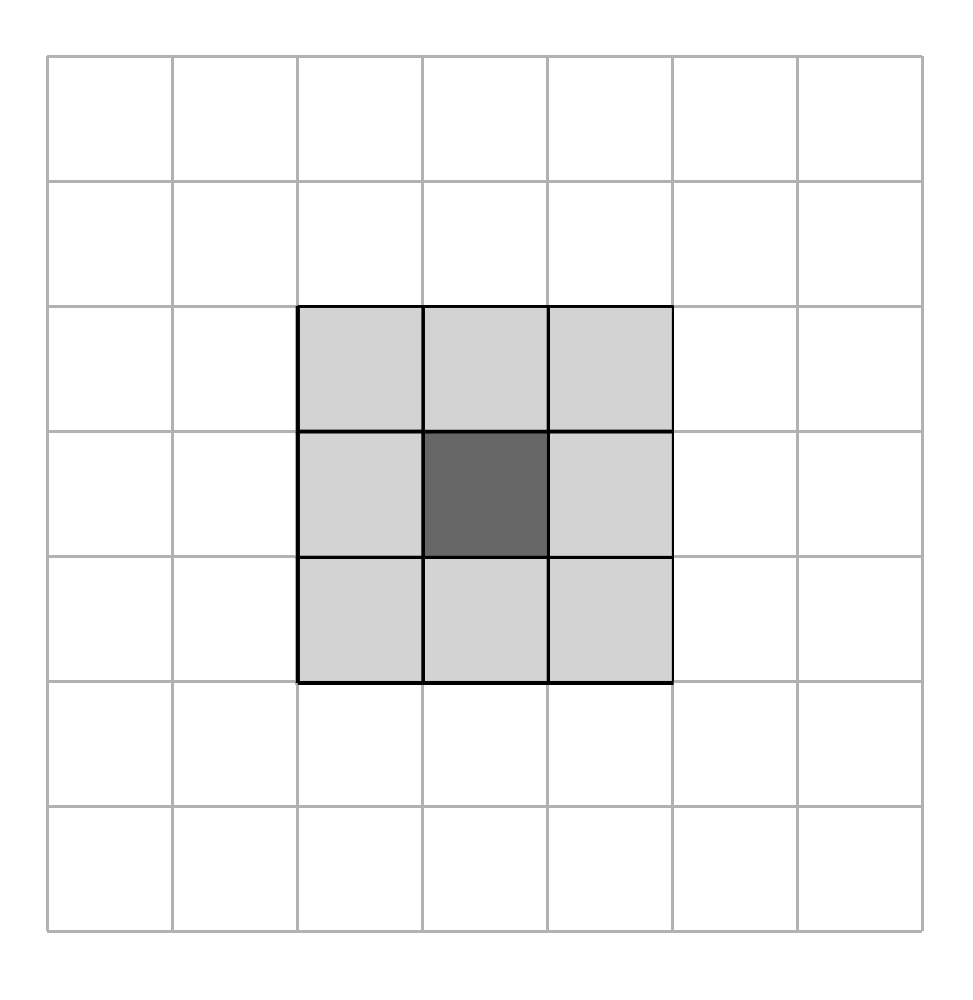}
    \caption{One-layer element patch, $k = 1$.}
  \end{subfigure}
  \hspace{1em}
  \begin{subfigure}{.45\textwidth}
    \centering
    \includegraphics[width=4cm]{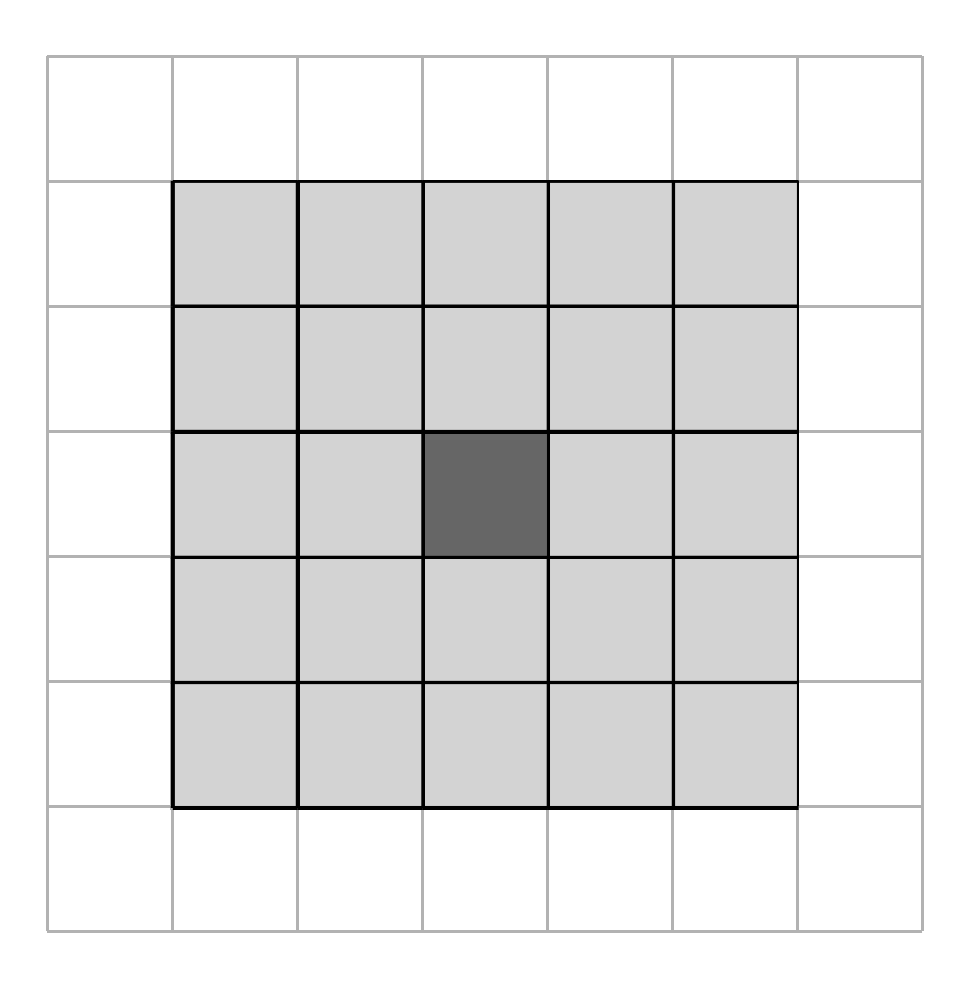}
    \caption{Two-layer element patch, $k = 2$.}
  \end{subfigure}
  \caption{Illustration of $k$-layer element patches for quadrilateral
    elements. Dark gray is $T$.  Light gray is $U_k(T)$.}
  \label{fig:patch}
\end{figure}
We further define localized fine spaces
\begin{equation*}
  \Vf(U_k(T)) = \{ v \in \Vf \,:\, v|_{\Omega \setminus U_k(T)} = 0\},
\end{equation*}
consisting of fine functions which are zero outside element
patches. Throughout the paper, localized quantities are subscripted
with the patch size $k$.

Instead of solving \eqref{eq:vms_corrections}, we compute the
operators $\Q_k = \sum_{T \in \triH} \Q_{k,T}$ and
$\R_k = \sum_{T \in \triH} \R_{k,T}$, with $\Q_{k,T}$ and $\R_{k,T}$ defined by
\begin{equation}
\begin{aligned}
  \label{eq:lod_corrections}
  (A \nabla \Q_{k,T} v, \nabla \vf) &= (A \nabla v, \nabla \vf)_T, \\
  (A \nabla \R_{k,T} f, \nabla \vf) &= (f, \vf)_T, \\
\end{aligned}
\end{equation}
for all $\vf \in \Vf(U_k(T))$ and all $T \in \triH$.

We define the localized multiscale space $\Vmsk = V_H - \Q_k V_H$. Our localized
multiscale problem reads find $\umsk \in \Vmsk$, such that for all
$v_H \in V_H$,
\begin{equation*}
\begin{aligned}
  (A \nabla \umsk, \nabla v_H) & = (f, v_H) - (A \nabla g, \nabla v_H) - {} \\
  & \phantom{{}={}} (A \nabla \R_k f, \nabla v_H) + (A \nabla \Q_k g, \nabla v_H),
\end{aligned}
\end{equation*}
and the full solution for the second approximation is $u_k = \umsk + \R_k f - \Q_k g$.

\subsection{Lagging multiscale space}
In the third approximation we compute the localized element correctors
using a lagging coefficient ${\tilde A}_T$ rather than the true $A$. This
makes it possible to reuse correctors that have been computed at
earlier time steps, so that localized correctors only for a small
number of elements $T$ need to be recomputed.

We define the lagging localized corrector operators
${\tilde \Q}_{k} = \sum_T {\tilde \Q}_{k,T}$ and
${\tilde \R}_{k} = \sum_T {\tilde \R}_{k,T}$. The element corrector
operators ${\tilde \Q}_{k, T} v, {\tilde \R}_{k, T} f \in \Vf(U_k(T))$
are defined such that for all $\vf \in \Vf(U_k(T))$,
\begin{equation}
  \begin{aligned}
    \label{eq:lagged_corrections}
    ({\tilde A}_T \nabla {\tQ}_{k, T} v, \nabla \vf) &= ({\tilde A}_T \nabla v,  \nabla \vf)_T, \\
    ({\tilde A}_T \nabla {\tR}_{k, T} f, \nabla \vf) &= (f, \vf)_T.
  \end{aligned}
\end{equation}
Note that lagging coefficients ${\tilde A}_T$ are not necessarily the same
for all $T$.
\begin{example}[Relation between lagging coefficient and time steps]
  As an example, for the current time step $A = A^n$, for element $T'$
  the coefficient can be one time step old, i.e.\
  ${\tilde A}_{T'} = A^{n-1}$ and for $T''$ three time steps old,
  i.e.\ ${\tilde A}_{T''} = A^{n-3}$. That is, different lagging
  localized element correctors may be defined in terms of coefficients
  from different time steps in history.
\end{example}

In analogy with previous multiscale spaces, we define a lagging
multiscale space $\tVmsk = V_H - \tQ_k V_H$ and the problem is then to
find $\tumsk \in \tVmsk$, such that for all $v_H \in V_H$,
\begin{equation}
  \label{eq:lod_falsespace}
  (A \nabla \tumsk, \nabla v_H) = (f, v_H) - (A \nabla g, \nabla v_H) - (A \nabla \tR_k f, \nabla v_H) + (A \nabla \tQ_k g, \nabla v_H)
\end{equation}
and the full solution for the third approximation is ${\tilde u}_k = \tumsk + \tR_k f - \tQ_k g$.


\subsection{Lagging global stiffness matrix contribution}
\label{sec:method4}
The fourth approximation involves not only using a lagging multiscale
space, but also a lagging coefficient in the assembly of the global
stiffness matrix and right hand side. The rationale behind this is
that computing the integrals in the stiffness matrix and right hand
side for \eqref{eq:lod_falsespace} requires that all precomputed
element correctors are stored. To circumvent this, we propose the
following approximation. First we define a lagging bilinear form
$\tilde a$ (and its elementwise contributor $\tilde a_T$), based on
the same lagging coefficients ${\tilde A}_T$ as was used for the
multiscale space in the previous section,
\begin{equation}
\label{eq:final_a}
{\tilde a}(u, v) := \sum_{T \in \triH} {\tilde a}_T(u, v) := \sum_{T \in \triH} ({\tilde A}_T (\chi_{T}\nabla - \nabla \tQ_{k,T})\IH u, \nabla v)
\end{equation}
where $\chi_{T}$ is the indicator function for subset
$T \subset \Omega$. We also define a lagging linear functional
$\tilde L$ (and its elementwise contributor $\tilde L_T$),
\begin{equation}
  \label{eq:final_L}
  \begin{aligned}
  \tilde L(v) & := \sum_{T\in\triH} \tilde L_T(v) \\
  & := \sum_{T \in \triH} (f, v)_T - (A\nabla g, \nabla v)_T - ({\tilde A}_T \nabla \tR_{k,T} f, \nabla v) + ({\tilde A}_T \nabla \tQ_{k,T} g, \nabla v).
  \end{aligned}
\end{equation}
Then the problem is posed as to find $\humsk \in \tVmsk$, such
that for all $v_H \in V_H$,
\begin{equation}
  \label{eq:lod_falsespacecoef}
  {\tilde a}(\humsk, v_H) = \tilde L(v_H).
\end{equation}
The full solution for the final approximation step is then
$\huk = \humsk + \tR_k f - \tQ_k g$.

We note that \eqref{eq:lod_falsespacecoef} coincides with
\eqref{eq:lod_falsespace} when ${\tilde A}_T = A$ for all $T$. Also,
we note that the coefficients for the linear system can be computed
immediately after ${\tQ}_{k,T}$ and ${\tR}_{k,T}$ have been
computed. This means no correctors need to exist simultaneously.

This method is independent of the true coefficient $A$, if not
${\tilde A}_T = A$ for any $T$. In order to construct a numerical
method with control of the error from this approximation, we use error
indicators on the element correctors to determine whether they need to
be recomputed or not. Next, we define three computable error indicators,
$e_u$, $e_f$ and $e_g$, for the error introduced by using a lagging
coefficient.

\subsection{Error indicators}
\label{sec:errorindicators}
As can be seen in later
Sections~\ref{sec:error_laggingmultiscalespace} and
\ref{sec:lagging_global}, the differences
$|\Q_{k,T}v-\tQ_{k,T}v|_A$ for $v \in V_H$,
$|\R_{k,T}f-\tR_{k,T}f|_A$, and $|\Q_{k,T}g-\tQ_{k,T}g|_A$ constitute
the sources to the error in the approximation from using lagging
coefficients. In this section, we define three elementwise error
indicators ($e_{u,T}$, $e_{f,T}$, and $e_{g,T}$) and relate them to
the above differences in Lemma~\ref{lem:errorind}.

\begin{lemma}[Error indicators: definitions and bounds] The following bounds hold,
  \label{lem:errorind}
  \begin{equation*}
    \begin{aligned}
      |\Q_{k,T}v-\tQ_{k,T}v|_A & \le e_{u,T} |v|_{A,T}, &\text{for all } v \in V_H,\\
      |\R_{k,T}f-\tR_{k,T}f|_A & \le e_{f,T} \|f\|_{L^2(T)}, \\
      |\Q_{k,T}g-\tQ_{k,T}g|_A & \le e_{g,T} |g|_{A,T}, \\
    \end{aligned}
  \end{equation*}
  where
  \begin{equation*}
    \begin{aligned}
      e_{u,T} & = \max_{w|_T\,:\,w \in V_H, |w|_{A,T} = 1} \|({\tilde A}_T - A) A^{-1/2} (\chi_{T} \nabla w - \nabla \tQ_{k,T}w)\|_{L^2(U_k(T))}, \\
      e_{f,T} & = \frac{\|({\tilde A}_T - A) A^{-1/2} \nabla \tR_{k,T}f\|_{L^2(U_k(T))}}{\|f\|_{L^2(T)}}\qquad \text{or 0 if } \|f\|_{L^2(T)} = 0,\\
      e_{g,T} & = \frac{\|({\tilde A}_T - A) A^{-1/2} (\chi_{T} \nabla g - \nabla \tQ_{k,T}g)\|_{L^2(U_k(T))}}{|g|_{A,T}} \qquad \text{or 0 if } |g|_{A,T} = 0.
    \end{aligned}
  \end{equation*}
  We additionally define
  \begin{equation*}
    \begin{aligned}
      e_u  = \max_{T \in \triH} e_{u,T},\quad
      e_f  = \max_{T \in \triH} e_{f,T},\quad \text{ and }
      e_g  = \max_{T \in \triH} e_{g,T}.
    \end{aligned}
  \end{equation*}
\end{lemma}
\begin{proof}
  For any $v \in V_H$, let $z = \Q_{k,T}v - \tQ_{k,T}v$, then using
  \eqref{eq:lod_corrections} and \eqref{eq:lagged_corrections}, we get
  \begin{equation*}
    \begin{aligned}
      |z|^2_{A,U_k(T)} & = (A \nabla (\Q_{k,T}v - \tQ_{k,T}v), \nabla z)_{U_k(T)} \\
      & = (({\tilde A}_T - A) \nabla \tQ_{k,T}v, \nabla z)_{U_k(T)} - (({\tilde A}_T - A)\nabla v, \nabla z)_T \\
      & \le \|({\tilde A}_T - A) A^{-1/2}(\chi_T \nabla v - \nabla \tQ_{k,T}v)\|_{L^2(U_k(T))} \cdot |z|_{A,U_k(T)}. \\
    \end{aligned}
  \end{equation*}
  Then, clearly $e_{u,T}$ (if it exists) constitute the asserted
  bound.  The following inequality gives a bound for the norm
  being maximized in the definition of $e_{u,T}$ (assuming that $|w|_{A,T} = 1$),
  \begin{equation*}
    \begin{aligned}
      &\|({\tilde A}_T - A) A^{-1/2} (\chi_{T} \nabla w - \nabla\tQ_{k,T} w)\|_{L^2(U_k(T))} \\
      &\qquad \le \|({\tilde A}_T - A) A^{-1}\|_{L^\infty(T)} + {}\\
      &\qquad \phantom{{}\le{}} \|({\tilde A}_T - A) A^{-1/2} {\tilde A}_T^{-1/2}\|_{L^\infty(U_k(T))}\|A^{-1/2}{\tilde A}_T^{1/2}\|_{L^\infty(T)}.
    \end{aligned}
  \end{equation*}
  The maximum is thus attained and exists by the extreme value theorem.

  Similarly, for $z = \R_{k,T} f- \tR_{k,T} f$, we have
  \begin{equation*}
    \begin{aligned}
      |z|^2_{A,U_k(T)} & = (A \nabla (\R_{k,T}f - \tR_{k,T} f), \nabla z)_{U_k(T)} \\
      & = (({\tilde A}_T - A) \nabla \tR_{k,T} f, \nabla z)_{U_k(T)} \\
      & \le \|({\tilde A}_T - A) A^{-1/2} \nabla \tR_{k,T} f\|_{L^2(U_k(T))} \cdot |z|_{A,U_k(T)}, \\
    \end{aligned}
  \end{equation*}
  which motivates the definition of $e_{f,T}$ and the asserted
  bound. The result for $e_{g,T}$ holds analogously.
\end{proof}

Regarding the computation of these error indicators, both $e_{f,T}$
and $e_{g,T}$ are straight-forward to compute, being a ratio of two
computable norms. The error indicator $e_{u,T}$ is also easy to
compute. It is the square root of a Rayleigh quotient for a
generalized eigenvalue problem (where the restriction $|w|_{A,T} = 1$
removes the singularity of the denominator matrix):
\begin{equation*}
  {\bf B} {\bf x}_\ell = \mu_\ell {\bf C} {\bf x}_\ell
\end{equation*}
with the matrices
\begin{equation*}
  \begin{aligned}
  B_{ij} &= \big(({\tilde A}_T - A)^2 A^{-1} (\chi_{T} \nabla \phi_j - \nabla \tQ_{k,T} \phi_j), \chi_{T} \nabla \phi_i - \nabla \tQ_{k,T}\phi_i\big)_{U_k(T)},\\
  C_{ij} &= (A \nabla \phi_j, \nabla \phi_i)_{T},
  \end{aligned}
\end{equation*}
for all $i,j=1,\ldots,m-1$ where $m$ is the number of basis functions
in $T$ (i.e.\ one of them removed). The squared maximum $e_{u,T}^2$
corresponds to the maximum eigenvalue $\max_\ell \mu_\ell$.  We
emphasize that the matrices $\bf B$ and $\bf C$ are very small: the same size
as the number of degrees of freedom in the coarse element $T$ (minus
one for removing the constant), e.g., $2\times 2$ for 2D simplicial
meshes or $7 \times 7$ for 3D hexahedral meshes.

\subsubsection{Coarse error indicators}
\label{sec:coarse_indicators}
In order to compute the error indicators $e_{u,T}$, $e_{f,T}$, and
$e_{g,T}$ we need access to the true coefficient $A$ and lagging
correctors $\tQ_{k,T}\phi_i$, $\tR_{k,T}f$, and $\tQ_{k,T}g$ at the
same time. This implies all lagging correctors need to be saved in
order to compute the error indicators. Since the correctors in
practice are defined on patches of a fine mesh, and the patch overlap
can be substantial, the memory requirements for saving them might be
large.  In this section, we construct an additional bound that makes
it possible to discard the lagging correctors after they have been
computed.

We construct the following bound starting from the definition of $e_{u,T}$ in Lemma~\ref{lem:errorind},
\begin{equation}
\begin{aligned}
  \label{eq:eaT_coarse}
  e_{u,T}^2 
   & {}\le \sum_{\substack{T' \in \triH \\ T'\cap U_k(T)\ne\emptyset}} \max_{\substack{w|_T\,:\,w \in V_H,\\|w|_{A,T} = 1}} \|({\tilde A}_T - A) A^{-1/2} (\chi_{T} \nabla w - \nabla \tQ_{k,T}w)\|_{L^2(T')}^2 \\
   & {}\le \sum_{\substack{T' \in \triH \\ T'\cap U_k(T)\ne\emptyset}} \|\delta_T\|^2_{L^\infty(T')} \|A^{-1/2}{\tilde A_T}^{1/2}\|^2_{L^{\infty}(T)} \cdot{} \\
   & \phantom{{}\le \sum_{\substack{T' \in \triH \\ T'\cap U_k(T)\ne\emptyset}}{}} \cdot \max_{\substack{w|_T\,:\,w \in V_H,\\|w|_{{\tilde A}_T,T} = 1}} \|{\tilde A}_T^{1/2} (\chi_{T} \nabla w - \nabla \tQ_{k,T} w)\|^2_{L^2(T')} \\
    & =: E_{u,T}.
%
\end{aligned}
\end{equation}
where $\delta_T = ({\tilde A}_T - A) A^{-1/2} {\tilde A}_T^{-1/2}$, and we used that
\begin{equation*}
|w|_{{\tilde A}_T,T} \le \|A^{-1/2}{\tilde A_T}^{1/2}\|_{L^{\infty}(T)} |w|_{A,T}
\end{equation*}
in the last inequality. We further define $E_u = \max_{T \in \triH} E_{u,T}$.

The maximum in \eqref{eq:eaT_coarse} corresponds to a maximum
eigenvalue of a low-dimensional generalized eigenvalue problem, as was
the case for $e_{u,T}$ in Section~\ref{sec:errorindicators}. More
specifically, it is the square root of the maximum eigenvalue
\begin{equation}
  {\tilde \mu}_{T,T'} := \max_\ell \mu_\ell
\end{equation}
of ${\bf B} {\bf x}_\ell = \mu_\ell {\bf C} {\bf x}_\ell$ with the
matrices
\begin{equation*}
  \begin{aligned}
  B_{ij} &= \big({\tilde A}_T (\chi_{T} \nabla \phi_j - \nabla \tQ_{k,T} \phi_j), \chi_{T} \nabla \phi_i - \nabla \tQ_{k,T}\phi_i\big)_{T'},\\
  C_{ij} &= (\tilde A_T \nabla \phi_j, \nabla \phi_i)_{T},
  \end{aligned}
\end{equation*}
for $i,j=1,\ldots,m-1$, where $m$ is the number of basis functions in
$T$.  We note that the quantity ${\tilde \mu}_{T,T'}$ can be computed
directly after corrector $\tQ_{k,T}$ has been computed for the basis
functions in element $T$. Now, $\tQ_{k,T}$ does not need to be saved
for computing $E_{u,T}$ later, and it can be discarded. In particular,
the memory required for storing ${\tilde \mu}_{T,T'}$ (which, however,
is needed to compute $E_{u,T}$) scales like $\O(k^dH^{-d})$.

Still, ${\tilde A}_T$ needs to be available to compute
$\|A^{-1/2}{\tilde A_T}^{1/2}\|^2_{L^{\infty}(T)}$ and
$\delta_T$. This might not be a problem in applications where there is
a low-dimensional description of the coefficient, for example if the
coefficient is defined by a set of geometric shapes which can be
described by location, size, shape and so on.  In the section for
numerical experiments, we will study an example of upscaled two-phase
Darcy flow, where we illustrate a way to avoid saving $\tilde A_T$.

The error indicator $E_u$ can replace $e_u$ in all results and
algorithms in this work. Similar coarse error indicators can be
derived for $e_f$ and $e_g$.

\section{Error analysis}
\label{sec:error}
In this section we study the approximation error of the three
approximations $\uk$, $\tuk$ and $\huk$, and the inf-sup stability for
the systems yielding the solutions $\ums$, $\umsk$, $\tumsk$ and
$\humsk$. Finally, in Theorem~\ref{thm:main} in
Section~\ref{sec:lagging_global}, we present a bound on the error
$u - \huk$ of the full approximation.

We use $C$ to denote a constant that is independent of the regularity
of $u$, patch size $k$ and coarse mesh size $H$. It can, however,
depend on the contrast $\frac{\beta}{\alpha}$. The value of the
constant is not tracked between steps in inequalities. By the notation
$a \lesssim b$, we mean $a \le Cb$.

\subsection{Variational multiscale method}
Since the varitional multiscale formulation \eqref{eq:vms_low} is only
a reformulation of the original problem, without any approximations,
there is no error. However, the well-posedness of the formulation is
still of interest.

\subsubsection{Stability}
Uniqueness of a solution to \eqref{eq:vms_low} is guaranteed by an inf-sup
condition for $a$ on $\Vms$ and $V_H$,
\begin{equation*}
\begin{aligned}
  \inf_{w \in \Vms} \sup_{v \in V_H} \frac{|(A \nabla w, \nabla v)|}{|w|_{A} | v |_{A}}  & = \inf_{w \in V_H} \sup_{v \in V_H} \frac{|(A \nabla (w - \Q w), \nabla v)|}{|w - \Q w|_A |v|_A} \\
  & = \inf_{w \in V_H} \sup_{v \in V_H} \frac{|(A \nabla (w - \Q w), \nabla (v - \Q v))|}{|w - \Q w|_A |v|_A} \\
  & \ge \inf_{w \in V_H} \frac{|w - \Q w|^2_A}{|w - \Q w|_A |w|_A} \\
  & = \inf_{w \in V_H} \frac{|w - \Q w|^2_A}{|w - \Q w|_A |\IH(w - \Q w)|_A} \\
  & \ge C_{\Cint}^{-1}\alpha^{1/2}\beta^{-1/2} =: \gamma. \\
\end{aligned}
\end{equation*}
The existence inf-sup condition holds analogously. We let $\gamma$
denote the inf-sup stability constant and note that it is depends on
the contrast for general $\IH$. See \cite{HeMa17,PeSc16} for corrector
localization results independent of the contrast.

\subsection{Localized orthogonal decomposition}
For the error analysis of LOD we recite previous exponential decay
results (first presented in \cite{MaPe14}) of the localized corrector
operators by means of the following lemmas. For example, the proof in
\cite[Lemma 3.6]{HeMa14} is almost directly applicable here.
\begin{lemma}[Localization error]
  \label{lem:localization}
  Let $k > 0$ be a fixed integer and let $p_T \in \Vf$ be the solution of
  \begin{equation*}
    (A \nabla p_T, \nabla \vf) = F_T(\vf)
  \end{equation*}
  for all $\vf \in \Vf$, where $F_T \in V^*$ such that $F_T(\vf) = 0$
  for all $\vf \in \Vf(\Omega \setminus T)$. Furthermore, we let
  $p_{k,T} \in \Vf(U_k(T))$ be the solution of
  \begin{equation*}
    (A \nabla p_{k,T}, \nabla \vf) = F_T(\vf)
  \end{equation*}
  for all $\vf \in \Vf(U_k(T))$. Then there exists a constant
  $0 < \theta < 1$ that depends on the contrast but not on $H$ or the
  variations of $A$, such that
  \begin{equation*}
    \left|\sum_T p_T - p_{k,T}\right|^2_A \lesssim k^{d} \theta^{2k} \sum_T\left|p_{T}\right|^2_A.
  \end{equation*}
\end{lemma}
This lemma can be applied for the localization error of both
$\Q - \Q_k$ and $\R - \R_k$. In analogy with the definition of $\Q_k$
as a sum of $\Q_{k,T}$, we can define $\Q = \sum_{T\in\triH} \Q_T$
with $\Q_T = \Q_{\infty,T}$. Then for any $v \in V$, we can
identify $\Q_T v$ with $p_T$ and $\Q_{k,T} v$ with $p_{k,T}$ in
the lemma above (and similarly for $\R$).

\subsubsection{Stability}
Using Lemma~\ref{lem:localization}, we get the following result for $\Q v - \Q_k v$, with $v \in H^1(\Omega)$,
\begin{equation}
  \label{eq:decay_error}
  \left|\Q v - \Q_{k} v\right|^2_A \lesssim k^{d} \theta^{2k} \sum_T\left|\Q_{T}v\right|^2_A \lesssim k^{d} \theta^{2k} |v|^2_A.
\end{equation}
If in addtion $v \in V_H$, we can use the stability of $\IH$ and continue to get
\begin{equation*}
 |v|^2_A \lesssim |\IH (v - \Q v)|^2_A \lesssim |v - \Q v|^2_A.
\end{equation*}
Using the result above, we can derive an inf-sup constant for $a$ and
the pair of spaces $\Vmsk$ and $V_H$,
\begin{equation}
  \label{eq:stab_lod}
  \begin{aligned}
    &\inf_{\wmsk \in \Vmsk} \sup_{v \in V_H} \frac{|a(\wmsk, v)|}{|\wmsk|_A |v|_A} \\
    &  \qquad \ge \inf_{w \in V_H} \sup_{v \in V_H} \frac{|a(w - \Q w, v)| - |a(\Q w - \Q_k w, v)|}{(|w - \Q w|_A + |\Q w - \Q_k w|_A) |v|_A} \\
    &  \qquad \ge \inf_{w \in V_H} \sup_{v \in V_H} \frac{|a(w - \Q w, v)| - Ck^{d/2}\theta^k|w - \Q w|_A|v|_A}{(1 + Ck^{d/2}\theta^k)|w - \Q w|_A |v|_A} \\
    &  \qquad \ge \frac{\gamma - Ck^{d/2}\theta^k}{1 + Ck^{d/2}\theta^k} =: \gamma_k.
  \end{aligned}
\end{equation}
For sufficiently large $k$, there is a uniform bound
$\gamma_0 \le \gamma_k$. See \cite{ElGiHe14} for more details on
stability of this approximation.

\subsubsection{Error}
For arbitrary $u_I \in \Vmsk$, using the equations \eqref{eq:vms_low}
and \eqref{eq:lod_corrections}, we have for all $v \in V_H$,
\begin{equation*}
  \begin{aligned}
    (A\nabla (\umsk - u_I), \nabla v) &= (A\nabla (\ums - u_I), \nabla v) + {} \\
    &\phantom{{}={}}(A\nabla (\R f - \R_k f), \nabla v) - (A\nabla (\Q g - \Q_k g), \nabla v).
  \end{aligned}
\end{equation*}
The inf-sup condition for uniqueness above yields the following
approximation result, for arbitrary $u_I \in \Vmsk$,
\begin{equation*}
  \gamma_0 |\umsk - u_I|_A \le |\ums - u_I|_A + |\R f - \R_k f|_A + |\Q g - \Q_k g|_A.
\end{equation*}
In analogy with \eqref{eq:decay_error} we get the following result for $\R f - \R_k f$,
\begin{equation*}
  \begin{aligned}
    \left|\R f - \R_{k} f\right|^2_A & \lesssim k^{d} \theta^{2k} \sum_T\left|\R_{T} f\right|^2_A \lesssim k^{d} \theta^{2k} \|f\|^2_{L^2}.
  \end{aligned}
\end{equation*}

Recall that $u = \IH\ums - \Q\IH\ums + \R f - \Q g$ and
$u_k = \IH\umsk -\Q_k\IH\umsk+ \R_k f - \Q_k g$. Now, if we choose
$u_I = \IH \ums - \Q_k\IH\ums \in \Vmsk$, then
$\ums-u_I = -(\Q-\Q_k)\IH\ums$ and using the approximation result
we get
\begin{equation}
\begin{aligned}
  \label{eq:through_interpolation}
  & |\ums - \umsk|_A \\
  & \qquad \le |\ums - u_I|_A + |\umsk - u_I|_A \\
  & \qquad \le (1+\gamma_0^{-1})|\ums - u_I|_A + \gamma_0^{-1}(|\R f-\R_k f|_A + |\Q g-\Q_k g|_A)\\
  & \qquad \le (1+\gamma_0^{-1})|(\Q - \Q_k)\IH \ums|_A + \gamma_0^{-1}(|\R f-\R_k f|_A + |\Q g-\Q_k g|_A).\\
\end{aligned}
\end{equation}
Then, using $\IH \ums = \IH u$, interpolation stability
\eqref{eq:interpolation} and stability of the continuous problem, we
have for the full error
\begin{equation}
\begin{aligned}
  \label{eq:error_lod}
  |u - u_k|_A & \le |\ums - \umsk|_A + |(\R - \R_k) f|_A + |(\Q - \Q_k) g|_A\\
  & \le (1+\gamma_0^{-1})\left(|(\Q - \Q_k)\IH \ums|_A + |\R f-\R_k f|_A - |\Q g-\Q_k g|_A\right)\\
  & \lesssim (1+\gamma_0^{-1})k^{d/2} \theta^{k}\left(|\IH \ums|_A + \|f\|_{L^2} + |g|_A\right)\\
  & \lesssim (1+\gamma_0^{-1})k^{d/2} \theta^{k}\left(|u|_A + \|f\|_{L^2} + |g|_A\right)\\
  & \lesssim (1+\gamma_0^{-1})k^{d/2} \theta^{k}(\|f\|_{L^2} + |g|_A).\\
\end{aligned}
\end{equation}
This result was first shown in \cite{MaPe14} and is noteworthy, since
the error of the approximation decays exponentially with increasing
$k$, independently of the regularity of the solution $u$.



\subsection{Lagging multiscale space}
\label{sec:error_laggingmultiscalespace}
For this step, we use a lagging multiscale space $\tVmsk$ and need to
establish an inf-sup stability constant for $a$ with respect to
$\tVmsk$ and $V_H$. We will use the results from
Lemma~\ref{lem:errorind} both for deriving stability and the
approximation error. The following full corrector error can be
derived using Lemma~\ref{lem:errorind},
\begin{equation}
  \begin{aligned}
    |\Q_kw - \tQ_k w|^2_A & = \Big|\sum_{T} (\Q_{k,T}w - \tQ_{k,T} w)\Big|^2_A \\
    & \lesssim k^d \sum_T |\Q_{k,T}w - \tQ_{k,T} w|^2_{A,U_k(T)} \\
    & \le k^d e^2_u |w|^2_A. \\
  \end{aligned}
\end{equation}
The bounds $|\R_k f - \tR_k f|^2_A \lesssim k^d e^2_f \|f\|^2_{L^2}$ and
$|\Q_k g - \tQ_k g|^2_A \lesssim k^d e^2_g |g|^2_{A}$ hold similarly.

We note that if ${\tilde A}_T = A$, then
$e_{u,T} = e_{f,T} = e_{g,T} = 0$.  Obviously, updating a lagging
coefficient for an element corrector leads to no error for this
element corrector.

\subsubsection{Stability}
\label{eq:method_3_stability}
We can now derive an inf-sup constant for $a$ on $\tVmsk$ and $V_H$,
using similar techniques as in \eqref{eq:stab_lod},
\begin{equation*}
\begin{aligned}
 & \inf_{\twmsk \in \tVmsk} \sup_{v \in V_H} \frac{|a(\twmsk, v)|}{|\twmsk|_A |v|_A} \\
  & \qquad \ge \inf_{w \in V_H} \sup_{v \in V_H} \frac{|a(w - \Q_k w, v)| - |a(\Q_k w - \tQ_k w, v)|}{(|w - \Q_k w|_A + |\Q_k w - \tQ_k w|_A) |v|_A} \\
  & \qquad \ge \inf_{w \in V_H} \sup_{v \in V_H} \frac{|a(w - \Q_k w, v)| - Ck^{d/2}e_u|w - \Q w|_A|v|_A}{(1 + Ck^{d/2}e_u)|w - \Q_k w|_A |v|_A} \\
  & \qquad \ge \frac{\gamma_k - Ck^{d/2}e_u}{1 + Ck^{d/2}e_u} =: {\tilde \gamma}_k.
\end{aligned}
\end{equation*}
We note that $k$ enters the constant, but that it can be compensated
by a small $e_u$. Since $e_{u,T}$ is computable, a rule to recompute
all element correctors $T$ with $e_{u,T} \ge \TOL(k)$ for some small
enough $\TOL(k) = \mathcal{O}(k^{-d/2})$, will (after recomputation)
make ${\tilde A}_T = A$ and $e_{u,T} = 0$. This makes $e_u < \TOL(k)$.
Following this adaptive rule makes it possible to find a lower bound
$\tilde \gamma_0 \le \tilde \gamma_k$ for sufficiently large $k$ and
sufficiently small $\TOL$.

\subsubsection{Error}
Again, we get an approximation result from the inf-sup
stability. In complete analogy with \eqref{eq:through_interpolation}
and \eqref{eq:error_lod}, we get
\begin{equation}
\begin{aligned}
  \label{eq:error_falsespace}
  & |\uk - \tuk|_A \\
  & \qquad \le (1+\tilde \gamma_0^{-1}) \left(|(\Q_k - \tQ_k)\IH \umsk|_A + |\R_k f-\tR_k f|_A + |\Q_k g-\tQ_k g|_A\right)\\
  & \qquad \lesssim (1+\tilde \gamma_0^{-1})k^{d/2}\left(e_u|u_k|_A + e_f\|f\|_{L^2} + e_g|g|_A\right)\\
  & \qquad \lesssim (1+\tilde \gamma_0^{-1})\gamma_0^{-1}k^{d/2}\max(e_u,e_f,e_g)(\|f\|_{L^2} + |g|_A).\\
\end{aligned}
\end{equation}

\subsection{Lagging global stiffness matrix contribution}
\label{sec:lagging_global}
In the fourth approximation \eqref{eq:lod_falsespacecoef}, the
coefficients for the integration of the global stiffness matrix and
(parts of) the right hand side are also lagging.
\subsubsection{Stability}
We derive an inf-sup constant for $\tilde a$ (see \eqref{eq:final_a})
with respect to $\tVmsk$ and $V_H$,
\begin{equation*}
\begin{aligned}
 & \inf_{\twmsk \in \tVmsk} \sup_{v \in V_H} \frac{|\tilde a(\twmsk, v)|}{|\twmsk|_A |v|_A} \\
  & \qquad \ge \inf_{\twmsk \in \tVmsk} \sup_{v \in V_H} (|\twmsk|_A |v|_A)^{-1} \Big(|a(\twmsk, v)| - {} \\
  & \qquad \phantom{{}={}}\Big|\sum_{T \in \triH} \int_{U_k(T)} ({\tilde A}_T-A) ( \chi_{T} \nabla - \nabla \tQ_{k,T}) \IH \twmsk \cdot \nabla v\Big| \Big) \\
  & \qquad \ge \tilde \gamma_k - \inf_{w \in V_H} \sup_{v \in V_H} \frac{\sum_T e_{u,T} \|A^{1/2}\nabla w\|_{L^2(T)} \|A^{1/2}\nabla v\|_{L^2(U_k(T))}}{|w - \tQ_k w|_A |v|_A} \\
  & \qquad \ge \tilde \gamma_k - \inf_{w \in V_H} \frac{Ck^{d/2} e_u \|A^{1/2}\nabla w\|_{L^2}}{|w - \tQ_k w|_A} \\
  & \qquad = \tilde \gamma_k - \inf_{w \in V_H} \frac{Ck^{d/2} e_u \|A^{1/2}\nabla \IH (w-\tQ_kw)\|_{L^2}}{|w - \tQ_k w|_A} \\
  & \qquad \ge \tilde \gamma_k - Ck^{d/2} e_u =: \hat \gamma_k\\
\end{aligned}
\end{equation*}
Again, $k$ enters, but can be compensated by a small $e_u$ according
to the discussion in Section \ref{eq:method_3_stability}. Thus, there
is a bound $\hat \gamma_0 \le \hat \gamma_k$ for all sufficiently
large $k$.

\subsubsection{Error}
To study the error $|\tuk - \huk|_{A}$, we first note that
$|\tuk - \huk|_{A} = |\tumsk - \humsk|_{A}$, since the right hand side
and boundary condition corrections are the same in both cases. We form
the following difference from \eqref{eq:lod_falsespace} and
\eqref{eq:lod_falsespacecoef},
\begin{equation}
  \begin{aligned}
  & a(\tumsk, v_H) - {\tilde a}(\humsk, v_H) \\
  & \qquad = \sum_T (({\tilde A}_T -A) \nabla \tR_{k,T} f, \nabla v_H) - \sum_T (({\tilde A}_T - A) \nabla \tQ_{k, T} g, \nabla v_H).
  \end{aligned}
\end{equation}
Add and subtract $a(\humsk, v_H)$ and use Lemma~\ref{lem:errorind} to get
\begin{equation*}
  \begin{aligned}
  & |a(\tumsk - \humsk, v_H)| \\
  & \qquad = \Big|\sum_T(({\tilde A}_T - A)(\chi_T \nabla - \nabla \tQ_{k,T})\IH\humsk, \nabla v_H) + {}\\
  & \qquad \phantom{{}={}\big|} \sum_T (({\tilde A}_T -A) \nabla \tR_{k,T} f, \nabla v_H) - {} \\
  & \qquad \phantom{{}={}\big|} \sum_T (({\tilde A}_T - A) \nabla \tQ_{k, T} g, \nabla v_H)\Big|\\
  & \qquad \le \sum_T \Big(\|({\tilde A}_T - A)A^{-1/2}(\chi_T \nabla - \nabla \tQ_{k,T})\IH\humsk\|_{L^2(U_k(T))} + {} \\
  & \qquad \hphantom{{}= \sum_T \Big(} \|({\tilde A}_T -A)A^{-1/2} \nabla \tR_{k,T} f\|_{L^2(U_k(T))} + {} \\
  & \qquad \hphantom{{}= \sum_T \Big(} \|({\tilde A}_T -A)A^{-1/2} \nabla \tQ_{k,T} g\|_{L^2(U_k(T))}\Big) |v_H|_{A,U_k(T)} \\
  & \qquad \lesssim k^{d/2} \Big( \sum_T e_{u,T}^2 |\IH \humsk|_{A,T}^2 + e_{f,T}^2 \|f\|_{L^2(T)}^2 + e_{g,T}^2 |g|_{A,T} \Big)^{1/2} |v_H|_A \\
  & \qquad \le k^{d/2} \max(e_u, e_f, e_g) \big(|\IH \humsk|_{A} + \|f\|_{L^2} + |g|_{A} \big) |v_H|_A .
  \end{aligned}
\end{equation*}
Inf-sup stability for $a$ and $\tilde a$ finally gives, for any $v_H \in V_H$,
\begin{equation}
\begin{aligned}
  \label{eq:error_falsespacecoef}
  |\tumsk - \humsk|_A & \le {\tilde \gamma}_0^{-1} \frac{|a(\tumsk - \humsk, v_H)|}{|v_H|_A} \\
  & \lesssim {\tilde \gamma}_0^{-1} k^{d/2} \max(e_u, e_f, e_g) \big(|\IH \humsk|_{A} + \|f\|_{L^2} + |g|_{A} \big) \\
  & \lesssim {\tilde \gamma}_0^{-1} {\hat \gamma}_0^{-1} k^{d/2} \max(e_u, e_f, e_g) \big(\|f\|_{L^2} + |g|_{A} \big). \\
\end{aligned}
\end{equation}

We conclude this section by presenting the main theoretical result of
this paper. It gives a bound of the full error of $\huk$ (in energy
norm) in terms of the patch size $k$ and the error indicators $e_u$,
$e_f$, and $e_g$ defined in Lemma~\ref{lem:errorind}. This theorem
forms the basis for the implementation of a method that updates the
multiscale space adaptively while iterating through the sequence of
coefficients.
\begin{theorem}[Error bound for multiscale method with lagging coefficient]
  \label{thm:main}
  Assume $k$ is sufficiently large, so that $\gamma_k \ge \gamma_0$
  holds. Let $\TOL = c k^{-d/2}$ and (by recomputation of element
  correctors) $\max(e_u, e_g, e_f) \le \TOL$. Choose $c$ sufficiently
  small so that $\tilde \gamma_k \ge \tilde \gamma_0$, and
  $\hat \gamma_k \ge \hat \gamma_0$. Further, let $u$ solve
  \eqref{eq:continuous} and $\humsk$ solve
  \eqref{eq:lod_falsespacecoef}. Let
  $\huk = \humsk + \tR_k f - \tQ_k g$. Then
  \begin{equation*}
    \begin{aligned}
      |u - \huk|_A & \lesssim k^{d/2}(\theta^{k} + \TOL)(\|f\|_{L^2} + |g|_A),
    \end{aligned}
  \end{equation*}
  where the hidden constant depends on the contrast but is
  independent of mesh size $H$, patch size $k$ and regularity of the
  solution $u$.
\end{theorem}
\begin{proof}
  The estimate of the full error $|u - \huk|_A$ is obtained by
  combining \eqref{eq:error_lod}, \eqref{eq:error_falsespace}, and
  \eqref{eq:error_falsespacecoef}, and using the triangle
  inequality,
  \begin{equation*}
    \begin{aligned}
      |u - \huk|_A & \le |u - \uk|_A + |\uk - \tuk|_A + |\tuk - \huk|_A  \\
      & \lesssim  k^{d/2} \Big((1+\gamma_0^{-1}) \theta^{k} + ((1+\tilde \gamma_0^{-1})\gamma_0^{-1} + {\tilde \gamma}_0^{-1} {\hat \gamma}_0^{-1}) \max(e_u, e_f, e_g)\Big) \cdot \\
      & \phantom{{}\lesssim{}} \cdot(\|f\|_{L^2} + |g|_A)\\
      & \lesssim k^{d/2}(\theta^{k} + \max(e_u, e_f, e_g))(\|f\|_{L^2} + |g|_A)\\
    \end{aligned}
  \end{equation*}
  and finally using the assumed bounds of $e_u$, $e_f$ and $e_g$.
\end{proof}

\begin{remark}[Selecting parameters $H$, $k$ and $\TOL$]
  The coarse mesh size parameter $H$ is typically chosen based the
  desired accuracy of the computation. The localization parameter $k$
  is chosen to be proportional to $|\log(H)|$ guaranteeing a
  perturbation of the approximation of the order
  $H|\log(H)|^{d/2}$. Finally, $\TOL$ is chosen proportional to
  $H$. The resulting error bound in energy norm then reads $\lesssim
  |\log(H)|^{d/2}H(\|f\|_{L^2} + |g|_A)$.
\end{remark}

\section{Implementation}
\label{sec:implementation}
In this section, we present an algorithm for computing approximate
solutions to a sequence of problems as described by
\eqref{eq:mainproblem}. In a practical implementation we can not let
$V$ be an infinite dimensional space. We will assume that there is a
finite element space $V_h$ based on a mesh that resolves the
coefficient, which if used to solve \eqref{eq:continuous}, yields an
approximate solution $u_h$ with satisfactory small error
$|u-u_h|_A$. The analysis in the previous sections holds also if
replacing $V$ with $V_h$, however, the error estimates will then of
course be bounding $|u_h - \huk|_A$ instead of $|u - \huk|_A$. In the
end of the section we discuss the memory requirements of the
algorithm.

The key idea is that, as time $n$ progresses, we
do not update the full multiscale space, but only the parts where it
is necessary for a sufficiently small error. If $A^n$ only changes
slightly between two consecutive $n$, it is possible that many of the
element correctors ($\tQ_{k,T}v_H$, $\tR_{k,T}f$ and $\tQ_{k,T}g$)
based on lagging coefficients do not need to be recomputed. We use the
error indicators $e_u$, $e_f$ and $e_g$ to determine for which
elements to recompute correctors.
This results in an algorithm that is completely parallelizable over
$T$, except for the solution of the low-dimensional (posed in $V_H$)
global system. Even the assembly of the global stiffness matrix ${\bf K} = (K_{ij})_{ij}$ and
right hand side ${\bf b} = (b_i)_i$ can be done in parallel, as it becomes a reduction
over $T$.

The algorithm is presented in Algorithm~\ref{alg:full}. We denote by
$\phi_i \in V_H$, $i = 1,2,\ldots$ the finite element basis functions
spanning $V_H$.

\newcommand\mycommfont[1]{{#1}}
\SetCommentSty{mycommfont}
\begin{algorithm}[H]
  \SetKwInOut{Input}{Input}\SetKwInOut{Output}{Output}
  \SetKwComment{Comment}{eq.\ }{}
  \SetKwComment{CommentNo}{}{}
  \DontPrintSemicolon

  Pick $k$, $\TOL$ and number of time steps $N$\;
  Let ${\tilde A}_T \leftarrow A^1$ for all $T$\;
  Compute $\tQ_{k,T}\phi_i$, $\tR_{k,T}f$ and $\tQ_{k,T}g$ for all $T$ and $i$ \Comment*{\eqref{eq:lagged_corrections}}
  \For{$n=1,\ldots,N$}{
    Let $A \leftarrow A^n$ be the true coefficient for this time step\;
    \For{all $T$}{
      Compute $e_{u,T}$, $e_{f,T}$ and $e_{g,T}$ \CommentNo*{Lemma~\ref{lem:errorind}}
      \If{$\max(e_{u,T}, e_{f,T}, e_{g,T}) \ge \TOL$}{
        Update lagging coefficient ${\tilde A}_T \leftarrow A^n$ \;
        Recompute $\tQ_{k,T}\phi_i$, $\tR_{k,T}f$ and $\tQ_{k,T}g$ \Comment*{\eqref{eq:lagged_corrections}}
      }
      Update $K_{ij} \mathrel{+}= \tilde a_T(\phi_j, \phi_i)$ using $\tQ_{k,T}\phi_i$ \Comment*{\eqref{eq:final_a}}
      Update $b_{i} \mathrel{+}= \tilde L_T(\phi_i)$ using $\tR_{k,T}f$ and $\tQ_{k,T}g$ \Comment*{\eqref{eq:final_L}}
    }
    Solve for $\humsk$, by computing ${\bf K}^{-1}{\bf b}$ \Comment*{\eqref{eq:lod_falsespacecoef}}
    Compute $\huk = \humsk + \tR_k f - \tQ_k g$ if needed
  }
 \caption{Main algorithm}
 \label{alg:full}
\end{algorithm}
Note that the {\bf if}-statement in this algorithm together with
properly chosen $k$ and $\TOL$, ensures that the conditions for
Theorem~\ref{thm:main} are fulfilled. The numerical experiment in
Section~\ref{sec:full_experiment} investigates the relations between
the error and $\TOL$ and the fraction of recomputed element
correctors.

The memory required to perform the main algorithm grows with $k$ in
the following manner. Suppose $\O(h^{-d})$ is the number of elements
in the fine discretizations, as is the case for quasi-uniform
meshes. To compute $e_u$, $e_f$ and $e_g$, we need to keep $\tilde
A_T$, $\tQ_{k,T}\phi_i$, $\tR_{k,T}f$, and $\tQ_{k,T}g$ between the
iterations in Algorithm~\ref{alg:full} (see
Lemma~\ref{lem:errorind}). Since the patches $U_k(T)$ overlap by
$\O(k^{d})$ coarse elements, the amount of memory required between the
iterations scales like $\O(k^{d}h^{-d})$. In high dimensions for very
fine meshes, the amount of memory needed for storage can become a
limitation. Depending on the application, it is possible to reduce the
memory requirements. Below we give two examples of such applications.
\begin{figure}[h]
  \centering
  \includegraphics{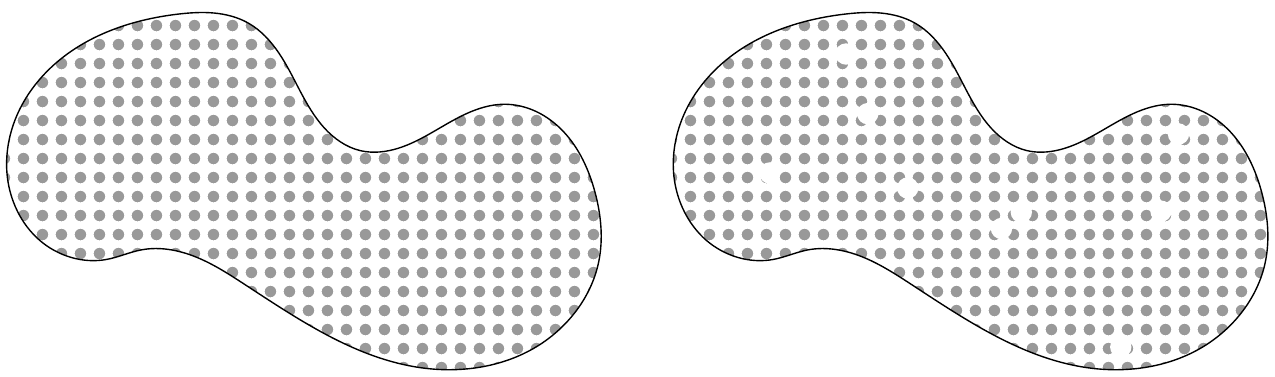}
  \caption{Defects in composite materials. Gray and white areas
    symbolizes two different values of the coefficient. Left: Reference material, lagging
    coefficient ${\tilde A}_T$. Right: Defect material, true coefficient $A$. }
  \label{fig:omega}
\end{figure}
\begin{example}[Defects in composite materials]
  For simulations on weakly random materials \cite{BrAn11}, we
  consider the coefficient of a reference material and a material with
  random defects shown in Figure~\ref{fig:omega}. If each ball in this
  material has a certain low probability to be missing (a localized
  point defect), the proposed method can be used to solve the model
  problem with the defect material on the right (true coefficient)
  using correctors precomputed on the reference material on the left
  (lagging coefficient). In sample based methods for stochastic
  integration (e.g.\ Monte Carlo), the proposed method for determining
  what correctors to recompute can reduce the computational cost for
  the full simulation.

  The lagging coefficient $\tilde A_T$ in this example is the single
  reference coefficient and thus the same for all $n$ and all
  $T$. Because of this, no additional memory is required to store
  lagging coefficients in this case. If we additionally use the (less
  efficient) coarse error indicators $E_{u,T}$, $E_{f,T}$, and
  $E_{g,T}$ presented in Section~\ref{sec:coarse_indicators}, our
  memory requirement scales with $\O(k^{d}H^{-d} + h^{-d})$ between
  the iterations in the algorithm.
\end{example}
\begin{example}[Two-phase Darcy flow]
  In a discretization of a two-phase Darcy flow system of equations
  (pressure and saturation equation) for an injection scenario, the
  permeability coefficient $A = A(s^n)$ varies over time $n$
  indirectly through the dependence on the saturation
  $s^n$. Typically, the change in saturation between time steps is
  localized to the front of the plume of the injected fluid. Thus,
  most corrector problems can be expected to be reused between
  iterations. An approach to reducing the memory requirements for the
  solution of this problem is revisited in detail in
  Section~\ref{sec:upscaling}.
\end{example}

\section{Numerical experiments}
\label{sec:numerical}
In all numerical experiments, we use $\Q_1$ Lagrange finite elements
in 2D or 3D on rectangular or rectangular cuboid elements. The degrees
of freedom are the values of the polynomial in the corners of the
element.
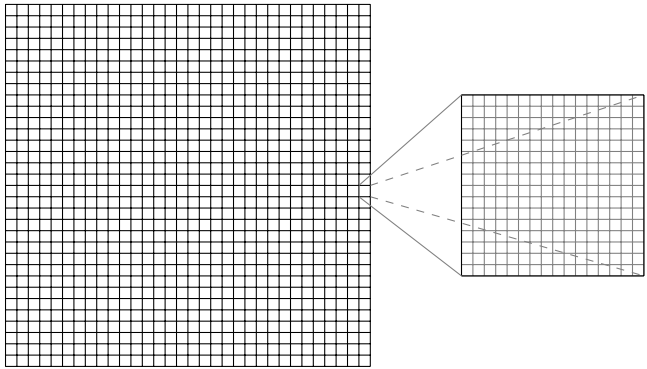
\begin{figure}[htb]
  \centering
  \begin{tikzpicture}[scale=0.15]
    \draw[step=1cm,black] (0,0) grid (32,32);
    \draw[xshift=40cm,yshift=8cm,step=1cm,gray] (0,0) grid (16,16);
    \draw[xshift=40cm,yshift=8cm,step=16cm,black] (0,0) grid (16,16);
    \draw[gray] (31,15) -- (40, 8);
    \draw[gray, dashed] (32,15) -- (56, 8);
    \draw[gray] (31,16) -- (40, 24);
    \draw[gray, dashed] (32,16) -- (56, 24);
  \end{tikzpicture}
  \caption{Illustrates a $32 \times 32$ coarse grid in 2D for
    $V_H$. To the right: A magnification of one coarse element revealing the
    fine discretization of $V$.}
  \label{fig:problem1_mesh}
\end{figure}

We define the interpolation operator $\IH$ to be used throughout the
experiments. Let $P_1$ denote the polynomials of no partial degree
greater than $1$, i.e.\ $\partial^2 p/\partial x^2 = 0$ for all
independent variables $x$ if $p \in P_1$. We define the broken finite
element space,
\begin{equation*}
  S_{H,\rm{b}} = \{ v \in L^2(\Omega) \,:\, v|_T \in P_1 \text{ for all } T \in \triH\}.
\end{equation*}
We denote by $\Pi_H$ the $L^2$-projection onto $S_{H,\rm{b}}$ and by
$E_H : S_{H,\rm{b}} \to V_H$ the boundary condition conforming node
averaging operator (Oswald interpolation operator), for all nodes in
$\triH$,
\begin{equation*}
  (E_H v)(x) = \begin{cases}
    0 &\text{if } x \in \Gamma_D, \\
    \operatorname{card}(T_x)^{-1}\sum_{T \in T_x} v|_T(x)&\text{otherwise}, \\
  \end{cases}
\end{equation*}
where $T_x = \{ T \in \triH \,:\, x \in \overline{T} \}$ and
$\operatorname{card}$ is the cardinality. Then we define
$\IH = E_H \circ \Pi_H$. This operator satisfies \eqref{eq:interpolation},
see e.g.\ \cite{DeEr12}.

\subsection{Experiments studying the effects of $k$ and $\TOL$}
\label{sec:full_experiment}
We let $\Omega = [0, 1]^2$,
$\Gamma_D = \{x \in \partial \Omega : x_1 = 0 \text{ or } x_1 = 1\}$,
$\Gamma_N = \partial \Omega \setminus \Gamma_D$, $f = 0$, $g = 1-x_1$,
and $A_{\rm b}$ as shown in Figure~\ref{fig:problem1_coef}. $A$ was constructed by taking a
uniform grid with $512 \times 512$ cells, and assigning each grid cell
a value $10^{c}$, where $c$ was drawn from a uniform distribution
between $[-2, 0]$, for each cell independently. Then, the values in
cells whose midpoint $x_{\rm m} = (x_{{\rm m},1}, x_{{\rm m},2})$ satisfied $15/32 \le x_{{\rm m},1} \le 1/2$
were set to $10^{-2}$. Finally, the values in
cells whose midpoint $x_{\rm m}$ satisfied $1/4 \le x_{{\rm m},2} \le 5/16$
were set to $1$.
\begin{figure}[h]
  \centering
  \begin{subfigure}[t]{0.45\textwidth}
    \centering
    \includegraphics[width=4cm, frame]{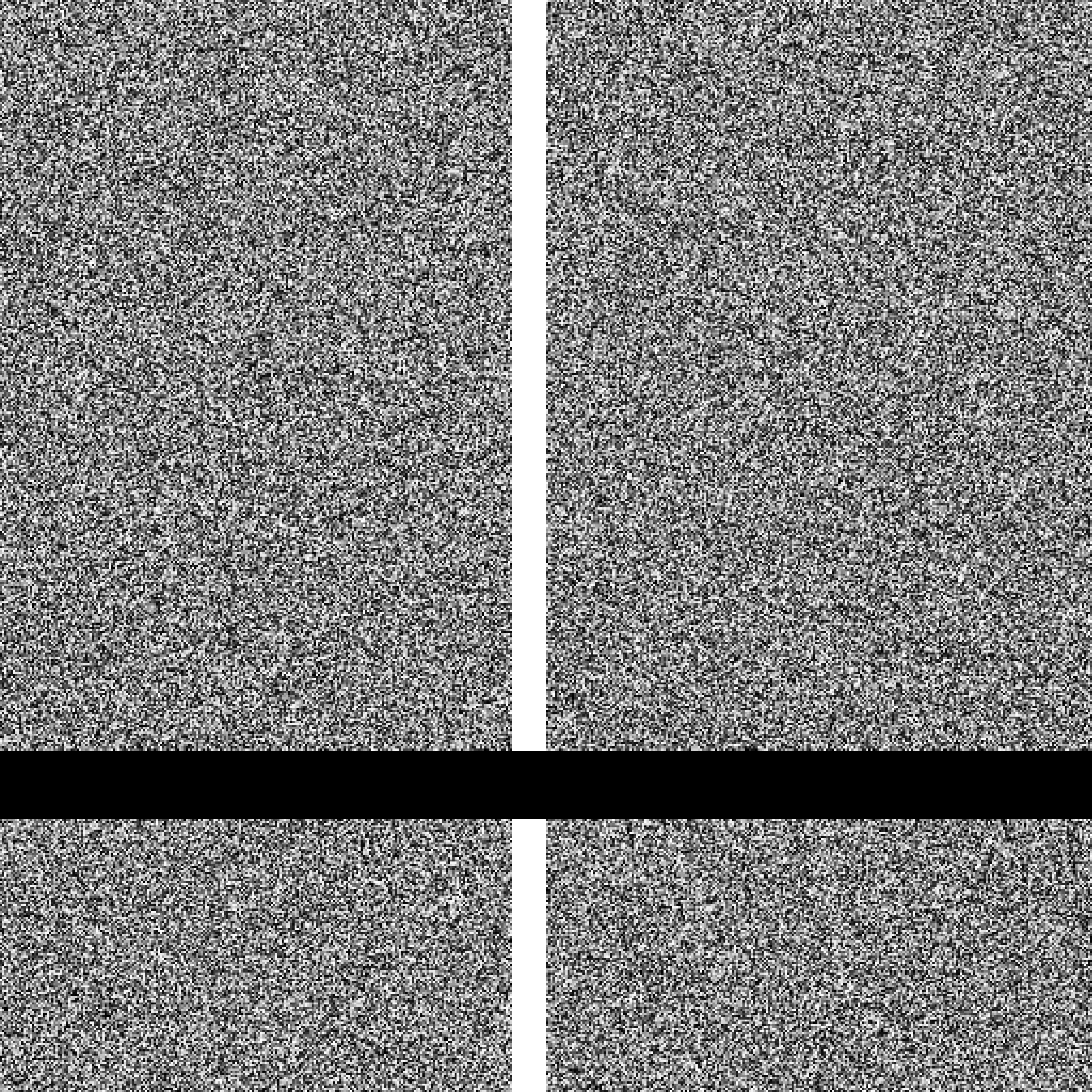}
    \caption{Coefficient $A_{\rm b}$. White means $A_{\rm b}=10^{-2}$,
      black means $A_{\rm b}=1$. The intensity of the grayscale scales
      linearly with $\log_{10}(A_{\rm b})$.}
    \label{fig:problem1_coef}
  \end{subfigure}
  \hspace{1em}
  \begin{subfigure}[t]{0.45\textwidth}
    \centering
    \includegraphics[width=4cm, frame]{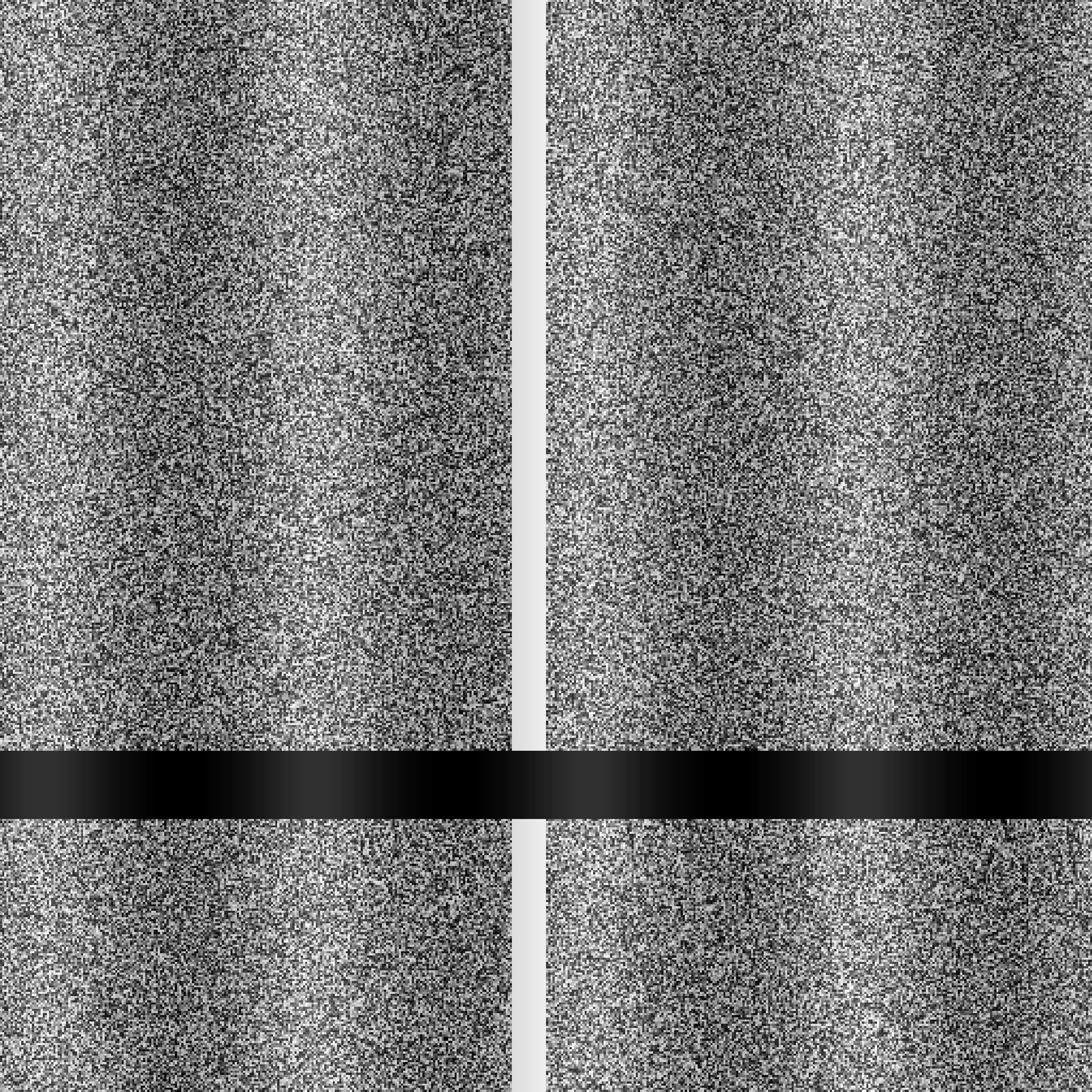}
    \caption{Coefficient $A_{n}$ for $n=13$. The color scale spans $[10^{-2}, 3]$.}
    \label{fig:problem1_coef2}
  \end{subfigure}
  \caption{Coefficients for the main algorithm experiment.}
\end{figure}
The space $V$ is discretized on a $\Q_1$ finite element space on a
uniform grid of size $512 \times 512$, see
Figure~\ref{fig:problem1_mesh}. The $V_H$ is chosen as a $\Q_1$ finite
element space on the coarse mesh shown in the same figure.

\subsubsection{Error decay with $k$}
First, we let $A = A_{\rm b}$ and solve for $u_k$ for $k = 1,2,3,$ and
$4$. We solve for $u$ on the fine mesh and use that as reference
solution.  The exponential convergence in terms of $k$ can be observed
in Figure~\ref{fig:problem1_kconv}.
\begin{figure}[h]
  \centering
  \includegraphics{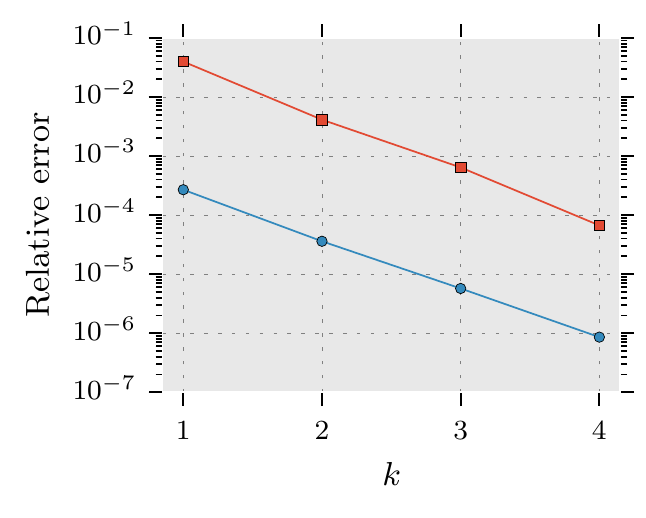}
  \caption{The data sets {\scriptsize $\square$} and $\circ$
    correspond to the (relative) errors $|u - u_k|_A$ and
    $\|\mathcal{I}_H (u - u_k) \|_{L^2}$, respectively.}
  \label{fig:problem1_kconv}
\end{figure}

\subsubsection{Error decay with $\TOL$}
Now we fix $k=3$ and define a sequence of
coefficients $A^n$ for $n=0,\ldots,127$,
\begin{equation*}
  A^n(x) = A_{\rm b}(x)\cdot(2+\sin(8\pi(x_1-n/128)))
\end{equation*}
This describes a perturbation of a factor up to $3$ over the full
domain, sweeping from the left to the right. We emphasize that the
difference $A_{n+1}-A_{n}$ is nonzero everywhere, which means a
strategy to determine which correctors to recompute is necessary.  We
use Algorithm~\ref{alg:full} to compute the approximate solution
${\hat u}_k$ for every time step $n$. A reference solution $u$ is also
computed. We do this for four values of $\TOL = 0.5, 0.1, 0.05$, and
$0.01$.

The (relative) error in energy norm versus the time step $n$ is
plotted to the left in Figure~\ref{fig:problem1_nstep}. The right plot
in the same figure shows the fraction of all element correctors
$\tQ_{k,T}$ that were recomputed in each time step. We note that the
error decreases with decreasing $\TOL$ as expected and that the
fraction of recomputed element correctors increase with decreasing
$\TOL$. Without an adaptive strategy, all element correctors would
have to be recomputed in every time step. See
Figure~\ref{fig:problem1_nstep_recomputedmap_0p1} for two maps over
the recomputed element correctors in time step $n=31$ for two
different values of $\TOL$.
\begin{figure}[h]
  \centering
  \includegraphics{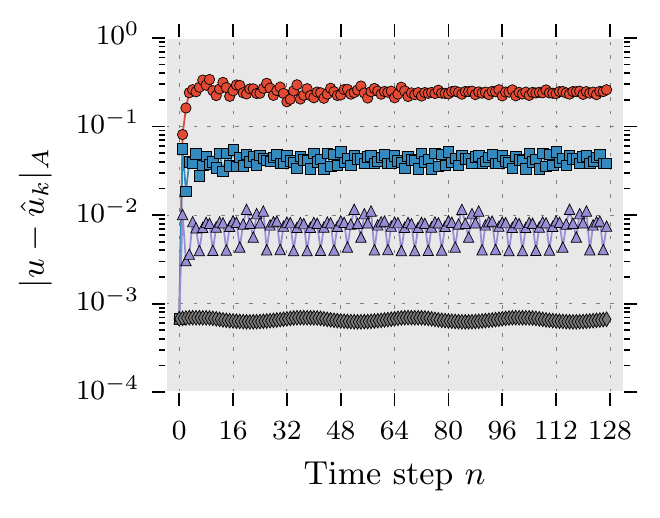}%
  \hspace{1em}%
  \includegraphics{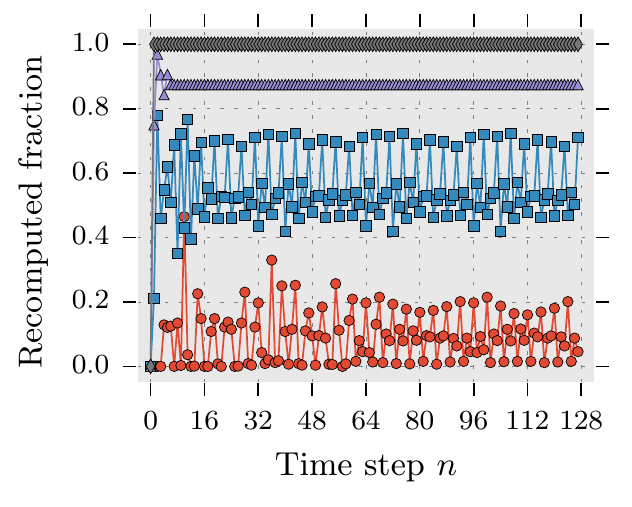}
  \caption{Error and recomputed fraction plots. The data sets $\circ$,
    {\scriptsize $\square$}, {\scriptsize $\triangle$}, and $\diamond$ correspond to $\TOL = 0.5, 0.1, 0.05$, and
    $0.01$, respectively.}
  \label{fig:problem1_nstep}
\end{figure}

\begin{figure}[h]
  \centering
  \includegraphics[width=5cm]{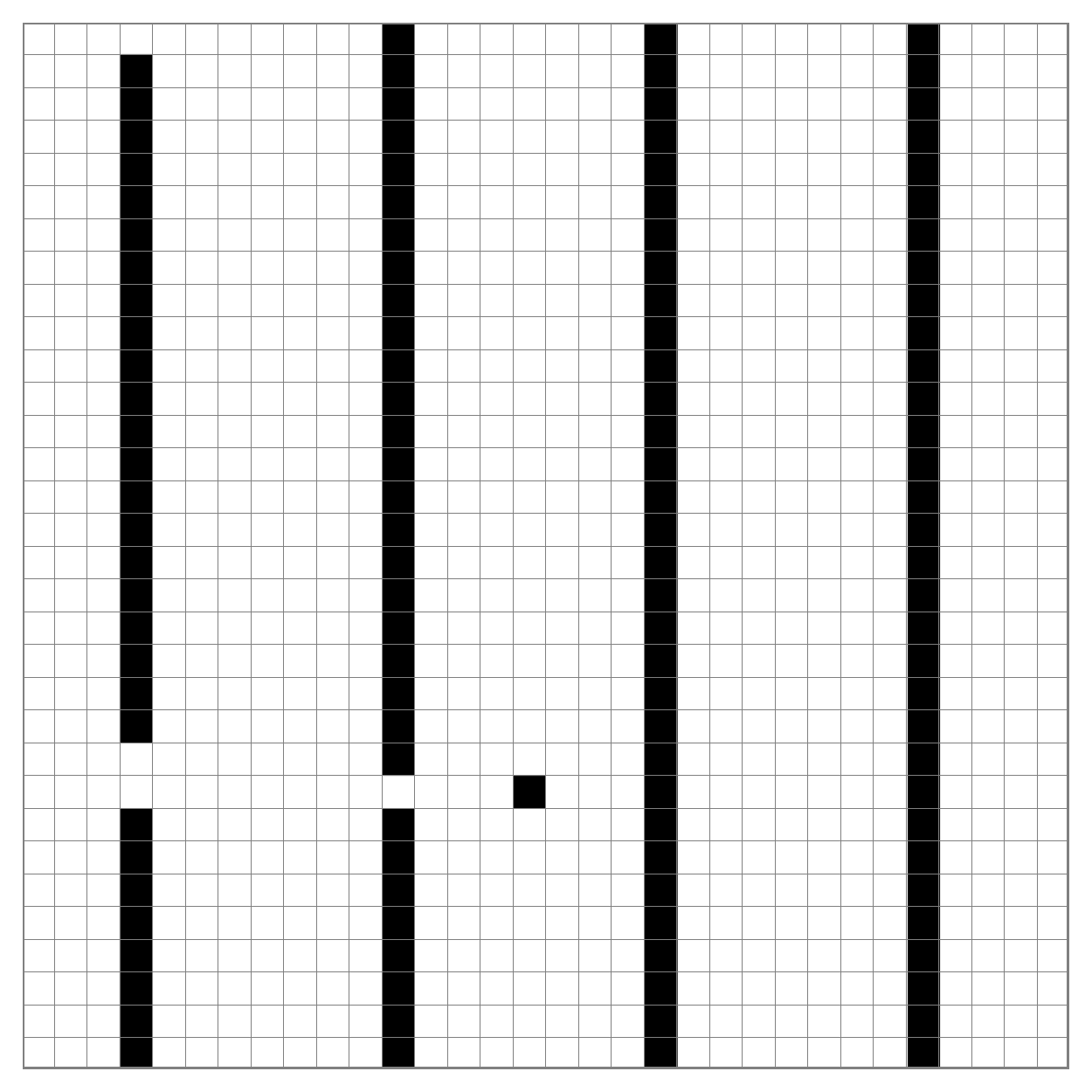}%
  \hspace{3em}%
  \includegraphics[width=5cm]{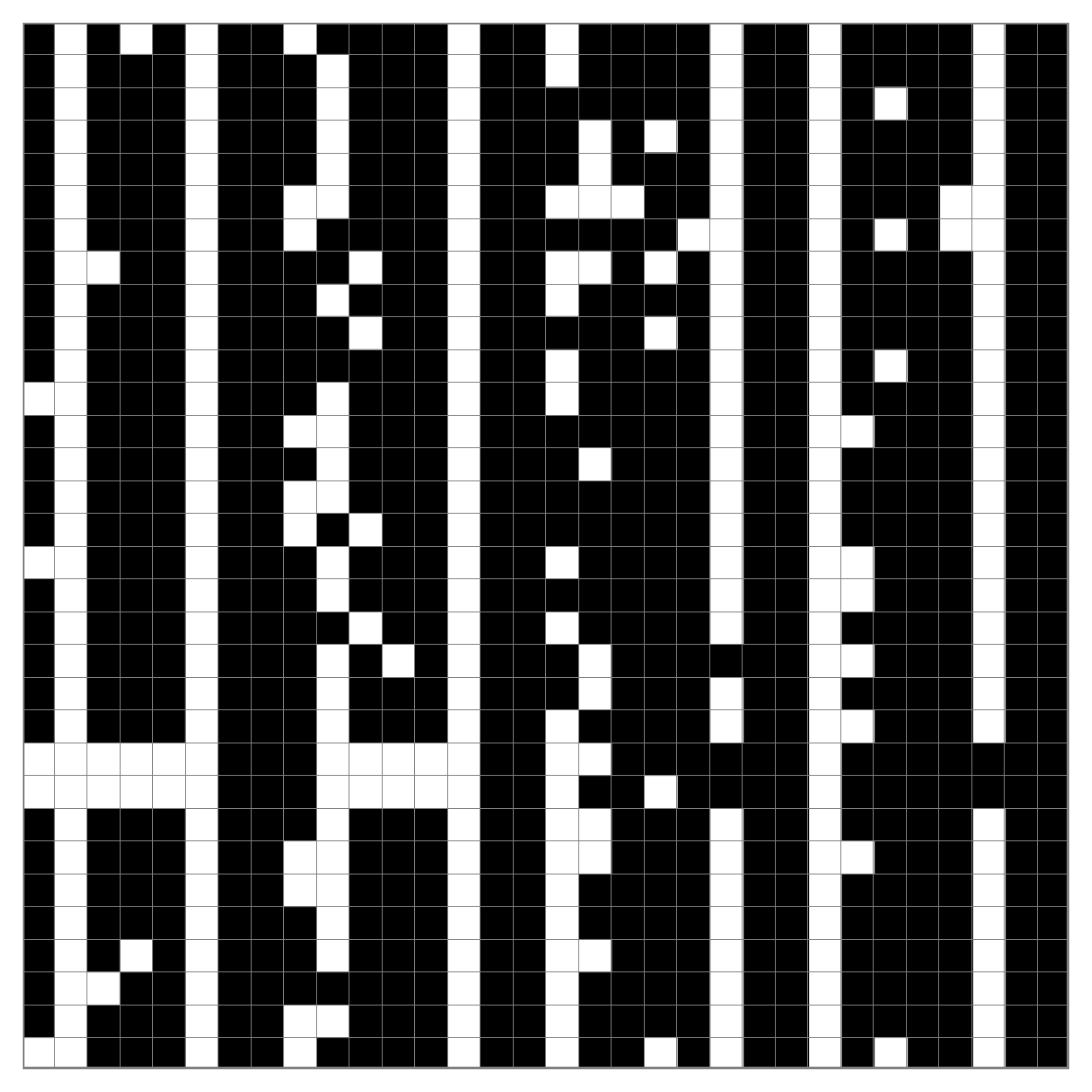}
  \caption{Time step $n=31$. Elements $T$ for which $\tQ_{k,T}$ were
    recomputed are colored black. Left and right figure show the
    recomputed element correctors for $\TOL = 0.5$ and $\TOL = 0.1$,
    respectively.}
    \label{fig:problem1_nstep_recomputedmap_0p1}
\end{figure}

\subsection{Low-memory Darcy flow upscaling algorithm}
\label{sec:upscaling}
In order to continue with two additional numerical experiments (in
Section~\ref{sec:darcy_experiments}), we describe an algorithm for
pressure solution upscaling for Darcy flows that reduces the space
complexity to $\O(k^dH^{-d} + h^{-d})$ (from $\O(k^{d}h^{-d})$). This is done by solving the
saturation equation on the coarse mesh and the pressure equation with
saturation dependent diffusion coefficient on a fine mesh using the
proposed adaptive multiscale method. This is possible in a situation
where the diffusion coefficient cannot be averaged on the coarse mesh,
but the saturation solution can. The low space complexity and the
possibility to parallelize the corrector computations enable the
solution of large-scale problems of this kind.

\subsubsection{A Two-phase Darcy flow model problem}
We consider the immiscible non-capillary two-phase Darcy flow problem
using the fractional flow formulation \cite{He97, NoCe12}. This leads
to a system of a coupled pressure and saturation equation
\begin{equation}
  \begin{aligned}
    \label{eq:twophase}
    -\div \lambda(s) K \nabla u & = f, \\
    \frac{\partial s}{\partial t} - \div \lambda_w(s) K  \nabla u & = f, \\
  \end{aligned}
\end{equation}
where space-time functions: $u$, $s$ and $f$ are pressure, saturation
for the wetting phase, and sources/sinks, respectively; space function
$K$ is intrinsic permeability; and nonlinear scalar functions:
$\lambda$ and $\lambda_w$ are total mobility and wetting phase
mobility, respectively. A common technique used for solving this
system is sequential splitting, where the pressure equation and
saturation equation are solved separately within a time step $n$. This
means that, as we iterate in time, we need to solve a sequence of
pressure equations with coefficient $A^n(x) = \lambda(s(x, t_{n-1}))
K(x)$.  Since the wetting saturation $s$ changes only significantly
along the plume front between time steps, we are in the setting where
consecutive differences in the coefficient is localized.

The permeability $K$ varies on a fine scale, requring these variations
to be resolved by a fine mesh with mesh size $h$ in order to obtain an
accurate pressure solution. We consider the case when the saturation
equation needs only be solved on a coarser mesh with mesh size $H > h$
to obtain a sufficiently accurate saturation solution.  We let the
fine mesh $\trih$ be a refinement of the coarse mesh $\triH$. The
pressure and saturation equations are solved sequentially: Given
initial data for the saturation, the pressure equation is solved. An
approximation of the coarse element face flux is computed and used to
solve for the next saturation using a zeroth order upwind
discontinuous Galerkin method with explicit Euler forward
time-stepping.

We use the same discretization scheme as in \cite{OdWhKvLa17}.  Let
$P_0(\triH)$ be the space of piecewise constants on the elements of
$\triH$. Let $\F_H$ denote set of faces of $\triH$. Each face
$F \in \F_H$ has a normal direction ${\bf n}_F$ (outward pointing for
boundary faces). We define the jump operator over face $F$ as
$\jump{v} = (v|_{T_1})|_F {\bf n}_{1}\cdot {\bf n}_F + (v|_{T_2})|_F
{\bf n}_{2}\cdot {\bf n}_F$,
where ${\bf n}_1$ and ${\bf n}_2$ are the outward pointing face
normals of the two elements $T_1$ and $T_2$ adjacent to $F$. Let
$\langle \cdot, \cdot \rangle_\omega$ denote the $L^2$-scalar product
when $\omega$ is $d-1$-dimensional. We let the flow be completely
driven by boundary conditions, i.e.\ $f = 0$.  We use the following
discretization for the saturation equation. Given
$s_H^{n-1} \in P_0(\triH)$, ${\sigma}^{n} \in L^{1}(\F_H)$, find
$s_H^{n} \in P_0(\triH)$ such that for all $r_H \in P_0(\triH)$,
\begin{equation}
  \begin{aligned}
    \label{eq:transport}
    \Delta t^{-1} (s_H^{n} - s_H^{n-1}, r_H) & = - \langle \psi(s_{H,\rm{upw}}^{n-1}) {\sigma}^{n}, \jump{r_H} \rangle_{\F_{H,I}} - {} \\
    &\phantom{{}={}} \langle \psi(s_H^{n-1}) {\sigma}^{n}, r_H\rangle_{\F_{H,\rm{out}}} - \langle \psi(s_B) {\sigma}^{n}, r_H\rangle_{\F_{H,\rm{in}}}.
  \end{aligned}
\end{equation}
Here ${\sigma}^{n}$ is an upscaled total flux quantity approximating (over face
$F$) ${\sigma^n}|_F \approx - {\bf n}_F \cdot \lambda(s) K \nabla (u + g)$; the sets $\F_{H,I}$,
$\F_{H,\rm{out}}$, and $\F_{H,\rm{in}}$ contain interior faces, Dirichlet
boundary faces with outgoing and ingoing flux, respectively; $s_B$ is
the saturation boundary condition; and $s_{H,\rm{upw}}^{n}$ is the
upwind saturation
\begin{equation}
  s_{H,\rm{upw}}^{n}|_F = \begin{cases}
    (s_{H}^{n}|_{T_1})|_F, \quad \text{if } {\sigma}^n \ge 0, \\
    (s_{H}^{n}|_{T_2})|_F, \quad \text{if } {\sigma}^n < 0, \\
    \end{cases}
\end{equation}
where $T_1$ and $T_2$ are adjacent to $F$ and ${\bf n}_F$ points from
$T_1$ to $T_2$; and the function $\psi(s) = \lambda_w(s)/\lambda(s)$
is the so called fractional flow function. The discretization of the
pressure equation is: find $u_h^{n} \in V_h$, so that for all $v_h \in V_h$,
\begin{equation}
  \begin{aligned}
    (A^n \nabla u^{n}_h, \nabla v_h) &= -(A^n \nabla g, \nabla v_h), \\
  \end{aligned}
\end{equation}
where $A^n = \lambda(s_H^{n-1}) K$. As suggested in \cite{OdWhKvLa17},
we define the non-conservative pre-flux ${\bar {\sigma}}^{n}$ be a
harmonic average of the (discontinuous) element face flux
$-{\bf n}_F \cdot A^n \nabla (u_h + g)$ at the faces. Then we use the
post-processing technique presented in the same paper (with
non-weighted minimization) to post-process ${\bar {\sigma}}^{n}$ and
obtain the conservative flux ${\sigma}^{n}$ used in the saturation
equation.

We make two observations on the information exchange between the two
equations when using this discretization:
\begin{enumerate}
\item In the pressure equation, we are only interested in the coarse
  scale saturation $s_H^{n-1}$ from the saturation equation.
\item In the saturation equation, we are only interested in the
  upscaled flux $\int_F \sigma^{n}$ from the pressure equation.
\end{enumerate}

\subsubsection{Coarse error indicators}
\label{sec:coarse_error}
The first observation allows us to compute $E_u$, $E_f$ and $E_g$ from
Section~\ref{sec:coarse_indicators}, without saving $\tilde A_T$.
Suppose that for element $T$, the lagging coefficient
${\tilde A}_T = A^m$ is from time step $m < n$, and $A = A^n$. We note
that $\delta_T$ is a coarse quantity, since fine-scale $K$ cancels,
\begin{equation}
\begin{aligned}
  \delta_T = ({\tilde A}_T - A) A^{-1/2} {\tilde A}_T^{-1/2} = \frac{\lambda(s_H^{m-1})-\lambda(s_H^{n-1})}{\sqrt{\lambda(s_H^{m-1})\lambda(s_H^{n-1})}}.
\end{aligned}
\end{equation}
Thus, to compute $\delta_T$ in \eqref{eq:eaT_coarse}, only ${\tilde
  \lambda}_{T,T'} := \lambda(s^{m-1}_H)|_{T'}$ for $T' \subset U_k(T)$
need to be saved from previous time steps. The memory required to
store ${\tilde \lambda}_{T,T'}$ behaves like $\O(k^dH^{-d})$. Also,
${\tilde \lambda}_{T,T'}$ can be used to compute $A^{-1/2}{\tilde
  A_T}^{1/2}$ in \eqref{eq:eaT_coarse}.

To summarize, no lagging fine scale information needs to be
stored. Only the coarse quantities ${\tilde \mu}_{T,T'}$ and
${\tilde \lambda}_{T,T'}$ need to be saved between iterations.

\subsubsection{Coarse face flux}
\label{sec:coarse_flux}
The second observation is that we only need the coarse element
face flux $\int_F \sigma^n$ for the saturation equation. Since this
quantity is defined on the coarse mesh, we precompute
\begin{equation}
  \begin{aligned}
    {\tilde \sigma}^n_{u,T,T',F,i} & = -\int_{F} {\bf n}_{F}\cdot \average{{\tilde A_T}} (\chi_T \nabla \phi_i - \nabla \tQ_{k,T} \phi_i)|_{T'}, \\
    {\tilde \sigma}^n_{fg,T,T',F} & = -\int_{F} {\bf n}_{F}\cdot \average{{\tilde A_T}}\nabla(-\tQ_{k,T}g + \tR_{k,T}f)|_{T'},
  \end{aligned}
\end{equation}
for all $T,T' \in \triH$, for all faces $F \subset \overline{T'}$ and
all basis functions $\phi_i$ with support in $T$. Here, we used the
harmonic average
$\average{v}|_F = 2\frac{(v|_{T_1})|_F (v|_{T_2})|_F}{(v|_{T_1})|_F +
  (v|_{T_2})|_F}$,
where $T_1$ and $T_2$ are the two elements adjacent to $F$, if $F$ is
an interior face, and $\average{v}|_F = 2(v|_{T})|_F$, where $T$ is
adjacent to $F$, if $F$ is a boundary face.  The memory required for
storing ${\tilde \sigma}^n_{u,T,T',F,i}$ and
${\tilde \sigma}^n_{fg,T,T',F}$ scales with $\O(k^d H^{-d})$, since
the two quantities are zero for all pairs $(T', T)$ except when
$T' \subset U_k(T)$.

If we now let the coarse component of the multiscale solution be expressed as
$\IH \huk = \sum_{i} \alpha_i \phi_i$, then we can compute the upscaled
non-conservative face flux by
\begin{equation*}
  {\bar \sigma}^n|_F = \frac{1}{2}\sum_{T,T',i}  \alpha_i {\tilde \sigma}^n_{u,T,T',F,i} + \frac{1}{2}\sum_{T,T'} {\tilde \sigma}^n_{fg,T,T',F}.
\end{equation*}
The final conservative face flux $\sigma^n|_F$ is then computed using
the post-process\-ing technique developed in \cite{OdWhKvLa17}.

We conclude this section by listing the upscaling algorithm
(Algorithm~\ref{alg:upscaling}) for this two-phase Darcy flow
problem. In this algorithm, the memory requirements are
$\O(k^dH^{-d}+h^{-d})$ (where $h^{-d}$ is for the coefficient $A$
which can be distributed on different computational nodes). This
allows for very refined fine meshes. Also, the coarse element loop is
still completely parallel, and this algorithm serves as a good
candidate for a scalable memory efficient upscaling algorithm.

\begin{algorithm}[H]
  \SetKwInOut{Input}{Input}\SetKwInOut{Output}{Output}
  \SetKwComment{Comment}{eq.\ }{}
  \SetKwComment{Commentsec}{sec.\ }{}
  \DontPrintSemicolon

  Pick $k$, $\TOL$ and number of time steps $N$\;
  Let ${\tilde A}_T \leftarrow \lambda(s^0_H)K$ for all $T$\;
  Compute $\tQ_{k,T}\phi_i$, $\tR_{k,T}f$ and $\tQ_{k,T}g$ for all $T$ and $i$ \Comment*{\eqref{eq:lagged_corrections}}
  Save only coarse quantities, $\tilde \lambda_{T,T'}$, ${\tilde \mu}_{T,T'}$, ${\tilde \sigma}^n_{u,T,T',F,i}$, and ${\tilde \sigma}^n_{fg,T,T',F}$ \;
  \For{$n=1,\ldots,N$}{
    Let $A \leftarrow \lambda(s^{n-1}_H)K$ be the true coefficient for this time step\;
    \For{all $T$}{
      Compute $E_{u,T}$, $E_{f,T}$ and $E_{g,T}$ (using ${\tilde \lambda}_{T,T'}$, and ${\tilde \mu}_{T,T'}$) \Comment*{\eqref{eq:eaT_coarse}}
      \If{$\max(E_{u,T}, E_{f,T}, E_{g,T}) \ge \TOL$}{
        Update lagging coefficient ${\tilde A}_T \leftarrow \lambda(s^{n-1}_H)K$ \;
        Recompute $\tQ_{k,T}\phi_i$, $\tR_{k,T}f$ and $\tQ_{k,T}g$ \;
        Save only, ${\tilde \lambda}_{T,T'}$, ${\tilde \mu}_{T,T'}$, ${\tilde \sigma}^n_{u,T,T',F,i}$, and ${\tilde \sigma}^n_{fg,T,T',F}$ \;
      }
      Update $K_{ij} \mathrel{+}= \tilde a_T(\phi_j, \phi_i)$ using $\tQ_{k,T}\phi_i$ \Comment*{\eqref{eq:final_a}}
      Update $b_{i} \mathrel{+}= \tilde L_T(\phi_i)$ using $\tR_{k,T}f$ and $\tQ_{k,T}g$ \Comment*{\eqref{eq:final_L}}
    }
    Solve for $\humsk$, by computing ${\bf K}^{-1}{\bf b}$ \Comment*{\eqref{eq:lod_falsespacecoef}}
    Compute resulting flux $\sigma^n|_F$ and solve for $s^n_H$ \Comment*{\eqref{eq:transport}}
  }
 \caption{Upscaling algorithm}
 \label{alg:upscaling}
\end{algorithm}

\subsection{Darcy flow upscaling numerical experiments}
\label{sec:darcy_experiments}
In the following two experiments, we investigate the properties of the
upscaling algorithm presented in the previous section. We pick the
following mobility functions $\lambda_w(s) = s^3$, $\lambda_n(s) =
(1-s)^3$, and $\lambda(s) = \lambda_w(s) +\lambda_n(s)$.

\subsubsection{2D random field data}
We let $\Omega = [0, 1]^2$,
$\Gamma_D = \{x \in \partial \Omega : x_1 = 0 \text{ or } x_1 = 1\}$,
$\Gamma_N = \partial \Omega \setminus \Gamma_D$, $f = 0$, $g = 1-x_1$.
We use a $512 \times 512$ rectangular grid as fine mesh for $V_h$ and a
$64 \times 64$ grid as coarse mesh for $V_H$. The permeability $K$ is
realized as a piecewise constant function on the fine mesh from a
lognormal distribution with exponential spatial correlation and
standard deviation $3$, i.e. for a fine element midpoint $x_{\rm m}$,
\begin{equation*}
  K(x_{\rm m}) = \exp(3 \kappa(x_{\rm m}))
\end{equation*}
where $\kappa(x_{\rm m}) \sim \mathcal{N}(0, 1)$ and covariance between points are
\begin{equation*}
  \operatorname{cov}(\kappa(x),\kappa(y)) = \exp\left( \frac{-\|x-y\|_2}{d} \right),
\end{equation*}
with correlation length $d=0.05$ and where $\|\cdot\|_2$ denotes
Euclidian norm. The initial saturation is set to $s^0 = 0$, and
boundary conditions are set to $s_B = 1$ on the left boundary, which
is the only boundary with ingoing flux. The number of time steps and
their size are set to $N = 2000$ and $\Delta t = N^{-1}$,
respectively.

The upscaling algorithm was run with this setup with $k=1,2,3$ and
$\TOL=0.4$, 0.2, 0.1, 0.05, 0.025, and 0.0125. A reference solution
$s^n_{H,\rm ref}$, where the pressure equation was solved on the fine
mesh using the $Q_1$ standard finite element method in every iteration
was computed. To illustrate the need for upscaling, we also computed
the pressure equation on the coarse mesh using the standard $\P_1$
finite elements. See Figure~\ref{fig:problem2_ms512} for plots over
the error in the saturaton solution at the final time step and average
fraction of recomputed correctors. In the error plot we can see that
both parameters $k$ and $\TOL$ affect the error in the chosen regimes.
We note from the recomputation plot that there is no dependency
between the fraction of recomputed element correctors and patch size
$k$. Figure~\ref{fig:problem2_solage} shows an example of the
saturation solution and the number of times correctors have been
computed.
\begin{figure}[h]
  \centering
  \includegraphics[trim=0.1cm 0.1cm 0.1cm 0.1cm, clip=true]{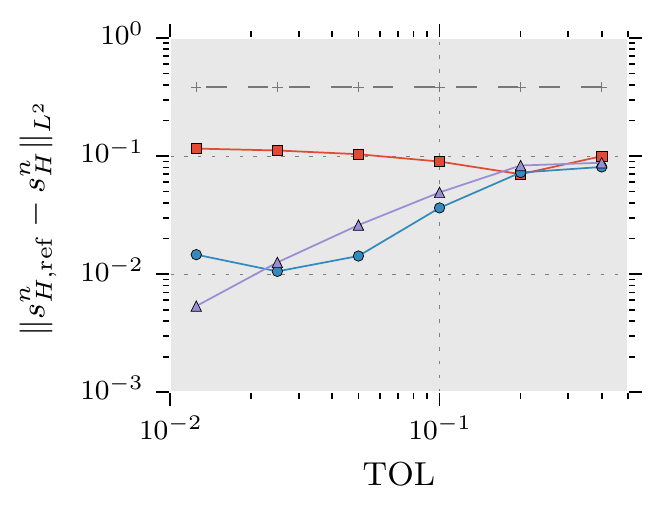}%
  \hspace{1ex}%
  \includegraphics[trim=0.1cm 0.1cm 0.1cm 0.1cm, clip=true]{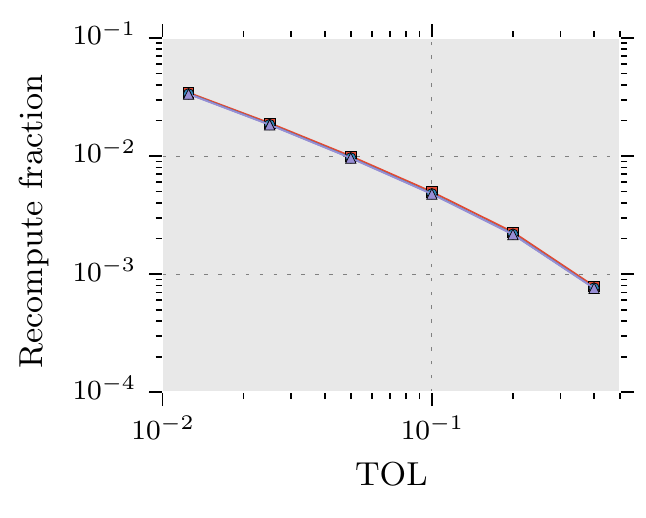}
  \caption{Error and recomputed fraction plots vs $\TOL$. The data
    sets $\circ$, {\scriptsize $\square$}, and {\scriptsize
      $\triangle$} correspond to $k = 1, 2$, and $3$,
    respectively. The dashed line ($+$) is the error obtained if using
    the standard finite element method on the coarse mesh.}
  \label{fig:problem2_ms512}
\end{figure}
\begin{figure}[h]
  \centering
  \includegraphics{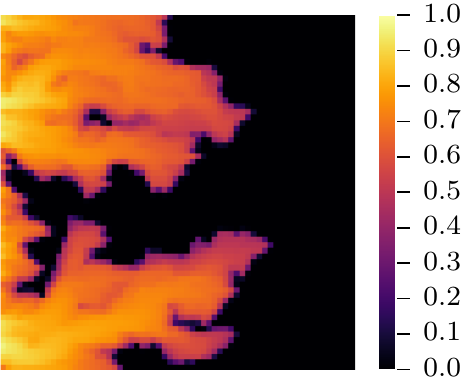}%
  \hspace{3em}%
  \includegraphics{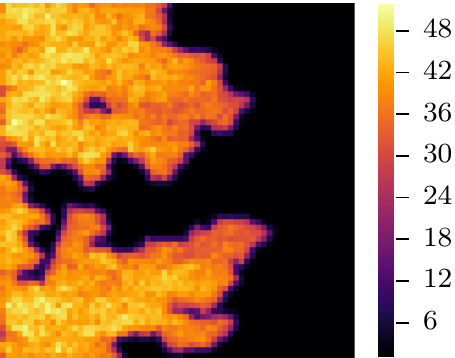}
  \caption{Time step $n = 2000$ for configuration $k=2$ and
    $\TOL=0.05$. Left: Solution $s^n_H$. Right: Recompuation count for
    each element (average is $20.3$ or 1\%).}
  \label{fig:problem2_solage}
\end{figure}

\subsubsection{3D random field data}
We let $\Omega = [0, 1]^3$,
$\Gamma_D = \{x \in \partial \Omega: x_1 = 0 \text{ or } x_1 = 1\}$,
$\Gamma_N = \partial \Omega \setminus \Gamma_D$, $f = 0$, $g = 1-x_1$.
We use a $128 \times 128 \times 128$ rectangular grid as fine mesh for
$V_h$ and a $16 \times 16 \times 16$ grid as coarse mesh for $V_H$. We
use a sample $\omega_{i,j,k,\ell}$ of independent uniformly
distributed random numbers between $0$ and $1$. The permeability $K$
is a piecewise constant function on the fine uniform mesh and is
defined by
\begin{equation*}
  K(x_{\rm m}) = \frac{1}{2^{21}}\prod_{i=1}^7 (1+\omega_{i,\lceil 2^i x_{{\rm m},1}\rceil, \lceil 2^i x_{{\rm m},2}\rceil, \lceil 2^i x_{{\rm m},3}\rceil})^3,
\end{equation*}
where $x_{\rm m} = (x_{{\rm m},1}, x_{{\rm m},2}, x_{{\rm m},3})$ are
the fine element midpoints, and $\lceil \cdot \rceil$ denotes the
ceiling function. See Figure~\ref{fig:3d_pictures} for the particular
realization used. The boundary conditions are set to $s_B = 0$ on the
boundary with ingoing flux, and the initial saturation is set to a
piecewise constant function $s^0_H$ on the coarse mesh with the
following values in the coarse element midpoints $x_{\rm m}$,
\begin{equation*}
  s^0_H(x_{\rm m}) = \begin{cases}
    1,\qquad \text{if } \|x_{\rm}-\frac{1}{2}(1,1,1)\|_2 \le \frac{1}{4}, \\
    0,\qquad \text{otherwise.}
    \end{cases}
\end{equation*}
The number of time steps are set to $N=200$ and the time step to
$\Delta t = 1$.

The upscaling algorithm was run for the three parameter combinations
I: $k=1$, $\TOL=0.1$, II: $k=2$, $\TOL=0.1$, and III: $k=1$,
$\TOL=0.01$.  Figure~\ref{fig:3d_pictures} gives an illustration of
the solution at $n=0$ and $n=200$. One of the images shows the
recomputed elements as blue boxes and we can see that many elements
are not recomputed. In this case, we were not able to compute a
reference solution using the available computational resources, but we
can estimate the sensitivity of the solutions with respect to the
parameters $k$ and $\TOL$. Let $s^{200}_{H, {\rm I}}$, $s^{200}_{H,
  {\rm II}}$, and $s^{200}_{H, {\rm III}}$ denote the saturation
solutions at time step $n=200$ for the parameter combinations I, II,
and III, respectively. We get
\begin{equation*}
  \begin{aligned}
    \text{vary }k: && \|s^{200}_{H, {\rm I}} - s^{200}_{H, {\rm II}}\|_{L^2(\Omega)} & = 0.0065, & \|s^{200}_{H, {\rm I}} - s^{200}_{H, {\rm II}}\|_{L^\infty(\Omega)} & = 0.0664, \\
    \text{vary }\TOL: && \|s^{200}_{H, {\rm I}} - s^{200}_{H, {\rm III}}\|_{L^2(\Omega)} & = 0.0019, & \|s^{200}_{H, {\rm I}} - s^{200}_{H, {\rm III}}\|_{L^\infty(\Omega)} & = 0.0242. \\
  \end{aligned}
\end{equation*}
These numbers suggest that the error due to localization (controlled
by the parameter $k$) dominates in this case.

\begin{figure}[h]
  \centering
  \begin{subfigure}{.4\textwidth}
    \centering
    \includegraphics[width=6cm,trim=0.1cm 0.1cm 0.1cm 0.1cm, clip=true]{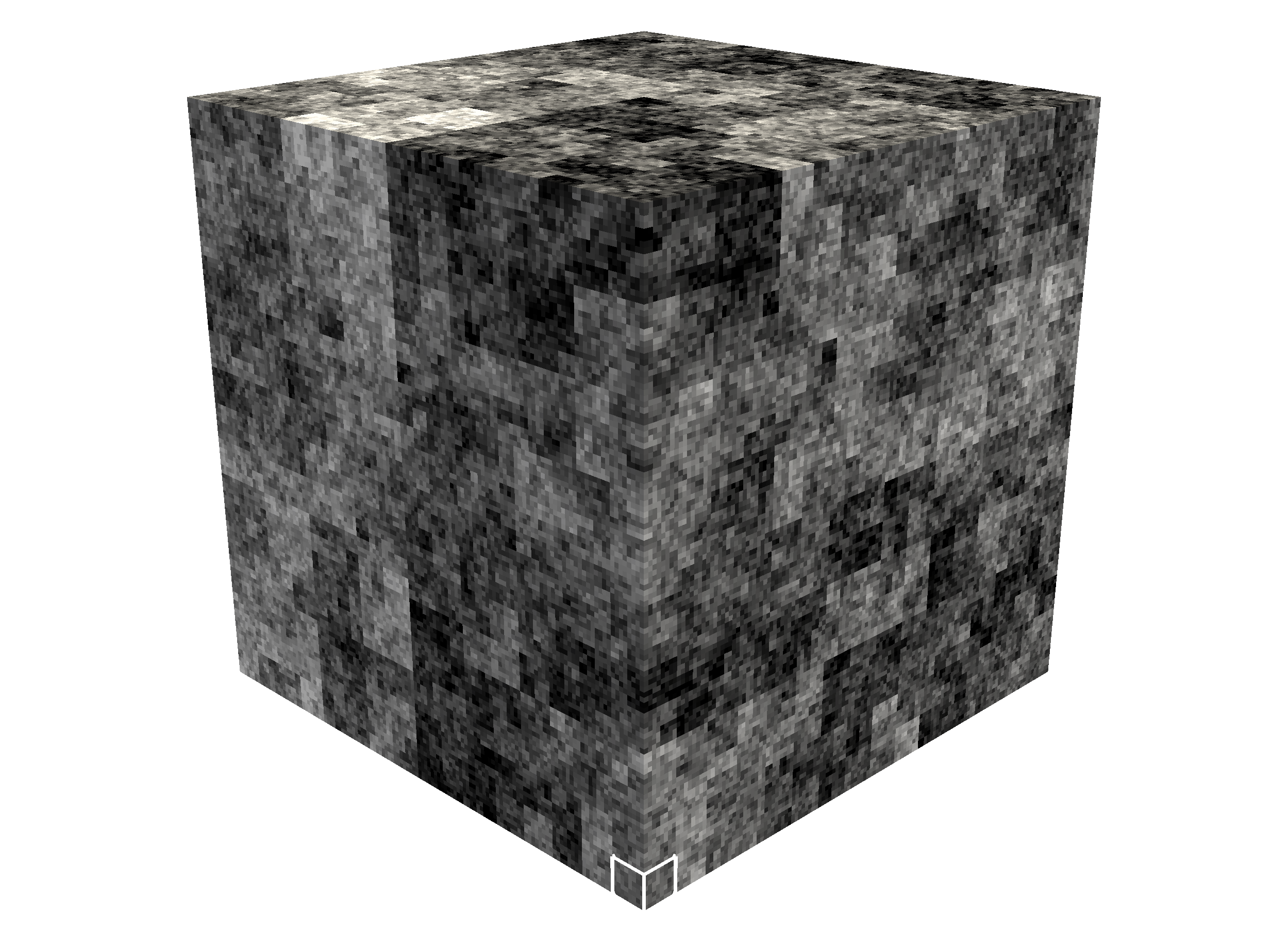}
    \caption{Coefficient on fine mesh. The coarse element size can be seen
      in the closest corner.}
  \end{subfigure}
  \hspace{1em}
  \begin{subfigure}{.4\textwidth}
    \centering
    \includegraphics[width=6cm,trim=0.1cm 0.1cm 0.1cm 0.1cm, clip=true]{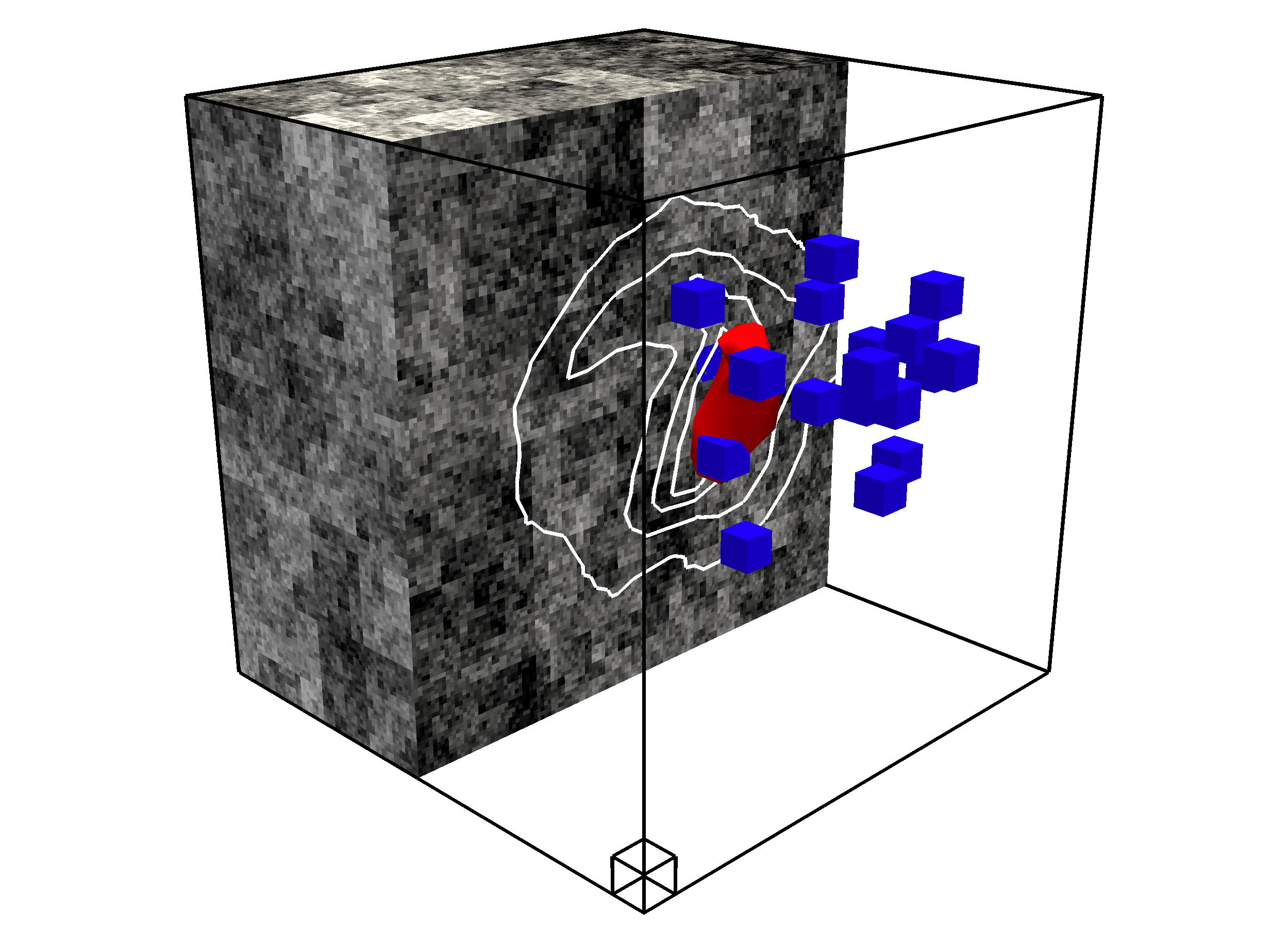}
    \caption{Time step $n = 200$. Blue elements have been recomputed this time step.}
  \end{subfigure}
  \hspace{1em}

  \begin{subfigure}{.4\textwidth}
    \centering
    \includegraphics[width=6cm,trim=0.1cm 0.1cm 0.1cm 0.1cm, clip=true]{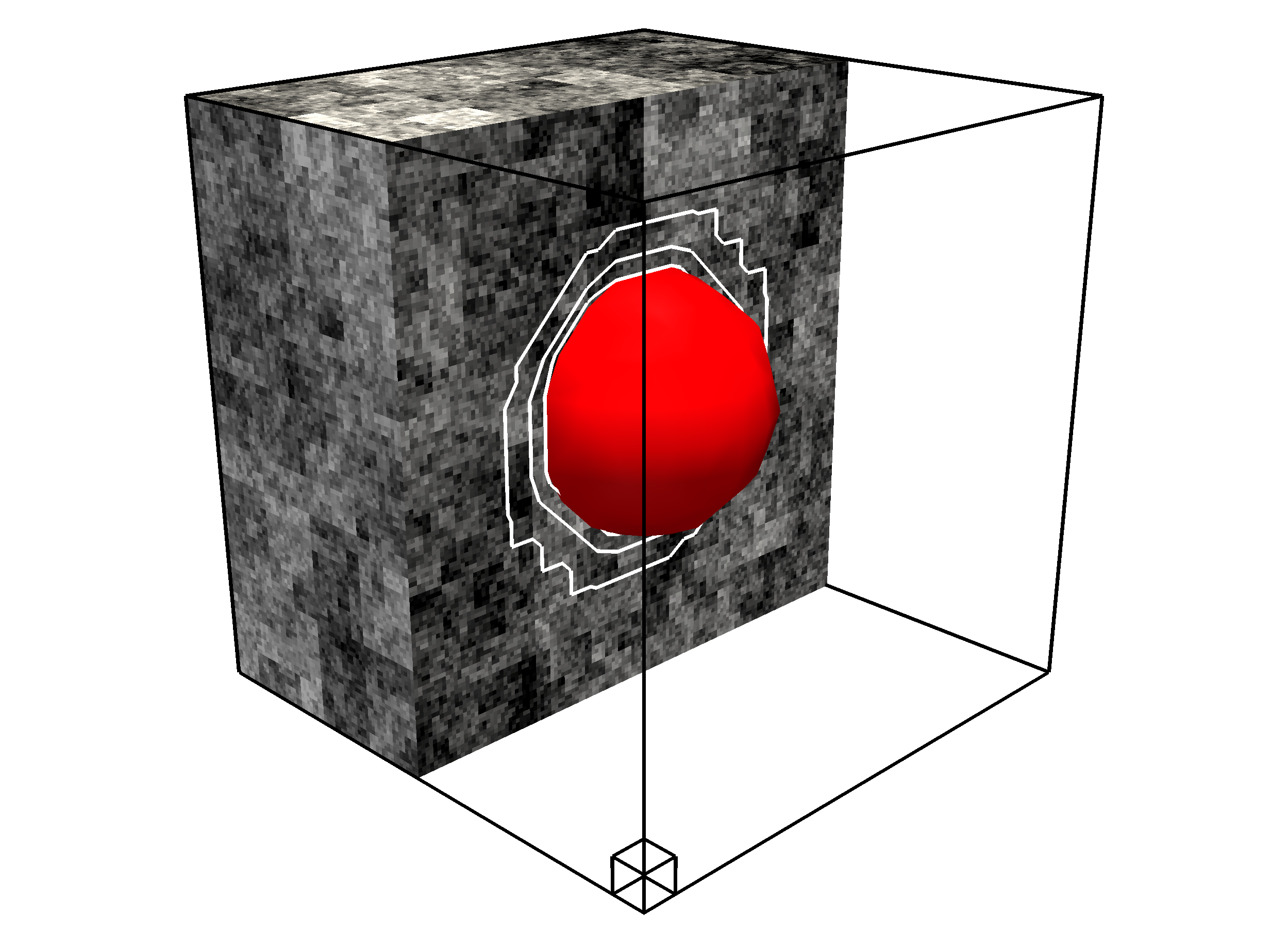}%
    \caption{A slice of the initial conditions at time step $n=0$.}
  \end{subfigure}
  \hspace{1em}
  \begin{subfigure}{.4\textwidth}
    \centering
    \includegraphics[width=6cm,trim=0.1cm 0.1cm 0.1cm 0.1cm, clip=true]{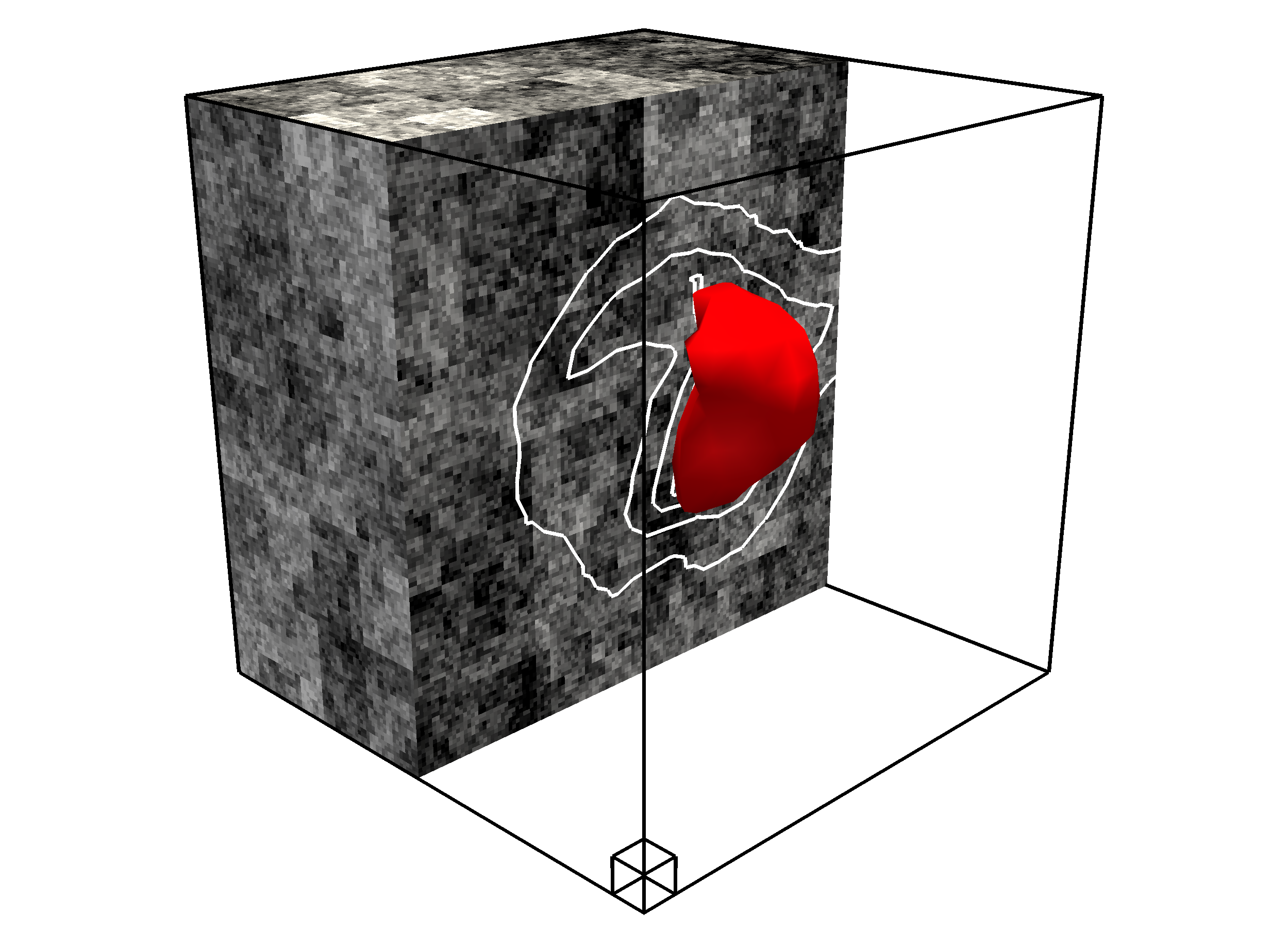}
    \caption{A slice of the solution at time step $n=200$.}
  \end{subfigure}
  \hspace{1em}
  \caption{The gray solid is the coefficient
    $1.42\times 10^{-6} \le K \le 0.46$. The red shape is the
    isosurface for $s=0.75$. The white lines are isolines for $s=0.1$,
    0.5, 0.75 and 0.9 at the slice, which is the plane $x_2=0.5$. The
    solutions and recomputed elements are shown for the case $k=1$,
    $\TOL = 0.1$}
  \label{fig:3d_pictures}
\end{figure}

\section{Conclusion}
Elliptic equations with similar rapidly varying coefficients occur for
instance in time-dependent problems for two-phase Darcy flow and in
stochastic simulations on defect composite materials. We consider a
sequence of elliptic equations, each with different coefficients
$A^n$, $n = 1,2,\ldots$. We define a method that computes and updates
an LOD multiscale space as we iterate through the coefficients. This
is done by the computation of localized element correctors that depend
on the coefficient in the vicinity of the element. These computations
can be performed completely in parallell. We derive error indicators
$e_{u,T}$, $e_{f,T}$, and $e_{g,T}$ that indicate whether or not to
update the corrector at element $T$ while iterating the sequence of
coefficients. By selecting a small enough tolerance $\TOL$ for the
error indicators, the multiscale space will keep its approximation
properties through the sequence of coefficients. It is shown
analytically and numerically that the error indicators bound the error
in energy norm of the solution. We present a memory efficient
upscaling algorithm for a particular application of two-phase Darcy
flows.

\FloatBarrier

\bibliographystyle{abbrv} \bibliography{references}

\end{document}